\documentclass[a4paper,sort&compress]{amsart}

\usepackage[utf8]{inputenc}
\makeatletter % change paragraph style to italic
\def\paragraph{\@startsection{paragraph}{4}%
  \z@\z@{-\fontdimen2\font}%
  \itshape}
\makeatother
\usepackage{xcolor} % e.g. for \color{gray}, also for \definecolor
\definecolor{RWTHblue}{RGB}{0,83,159}
\definecolor{RWTHbordeaux}{RGB}{161,16,53}
\usepackage{graphicx} % for \includegraphics
\usepackage[linkcolor=RWTHblue]{hyperref} % for clickable references
\hypersetup{
    hidelinks, % no boxes around links (only necessary if colorlinks = false
    colorlinks=true,
    hypertexnames=false, % somehow this fixes ``destination with the same identifier...'' by making hyperref dumber
    breaklinks,
    linkcolor=RWTHblue,
    citecolor=RWTHblue,
    urlcolor=RWTHblue
}
\usepackage[all]{hypcap} % jump to top of tables/figures... for hyperref
\usepackage{url} % enables urls in article 
\numberwithin{equation}{section} % count Thm. etc. inluding section number (apparently good to load before cleverref)
\usepackage[capitalize]{cleveref} % must be behind hyperref, before algorithm! (auto reference with theorem typename)

\usepackage[sort]{cite} % sort citations

\usepackage{amsfonts,amssymb,amsopn,amsmath,amsthm} % math stuff... for good measure
\usepackage{mathtools} % more math stuff that does not harm to have

\usepackage[font=footnotesize,textfont=it]{caption} % easy way to adapt the font size of all captions

\usepackage{enumitem} % more flexible lists
\usepackage{multirow} % for \multirow
\usepackage{placeins} % for \FloatBarrier (keeps floats like tables from passing... mostly)

\usepackage{algorithm} % should be behind cleverref, outer environment
\makeatletter
\renewcommand{\ALG@name}{\sc Algorithm} % nice font for algorithm
\makeatother
\usepackage{algorithmic} % inner algorithm environment
  
\renewcommand{\email}[1]{\protect\href{mailto:#1}{#1}} % stops ams from printing email at end of document

\title[Asymptotic Log-Det Rank Minimization Via (Alternating) IRLS]{Asymptotic Log-Det Rank Minimization Via (Alternating) Iteratively Reweighted Least Squares}

\author{Sebastian Kr\"amer}
\thanks{Institut f\"ur Geometrie und Praktische Mathematik,  RWTH Aachen University,  Templergraben 55,
52056 Aachen, Germany
  (\email{kraemer@igpm.rwth-aachen.de}, \url{https://www.igpm.rwth-aachen.de}).}

\theoremstyle{plain}
\newtheorem{theorem}{Theorem}[section]

\newtheorem{definition}[theorem]{Definition}
\newtheorem{proposition}[theorem]{Proposition}
\newtheorem{corollary}[theorem]{Corollary}
\newtheorem{lemma}[theorem]{Lemma}

\theoremstyle{definition}

\newtheorem{example}[theorem]{Example}
\newtheorem{experiment}[theorem]{Experiment}

\theoremstyle{remark} % not italic
\newtheorem{remark}[theorem]{Remark}

\crefname{experiment}{Experiment}{experiments}
\crefname{hypothesis}{Hypothesis}{Hypotheses}

\crefname{equation}{}{} % no Eq. in front of equation references

\newcommand{\R}{\ensuremath{\mathbb{R}}}
\newcommand{\N}{\ensuremath{\mathbb{N}}}
\DeclareMathOperator{\diag}{diag}
\DeclareMathOperator{\rank}{rank}

\newcommand{\IRLS}{\textsc{IRLS}}
\newcommand{\AIRLS}{{\sc AIRLS}}
\newcommand{\SALSA}{{\sc SALSA}}

\def\alg{{(\mathrm{alg})}}
\def\rs{{(\mathrm{rs})}}

\def\temporarycaptioncommand{}
\def\externaltablecaption#1#2#3{} % palceholder for supplement

\newcommand{\externaltable}[3]%
{%
\begin{table}[#1]%
\input{tables/#2_caption.tex}
\tiny
\begin{center}%
\input{tables/#2.tex}%
\end{center}%
\caption{\temporarycaptioncommand{} -- #3}
\end{table}\noindent
}

\def\barfigurecaption#1#2{} % placeholder for main
\def\buttonfigurecaption#1#2{} % placeholder for main

\newcommand{\pdffigure}[3]%
 {%
 \begin{figure}[#1]%
  \begin{center}%
    \includegraphics[width=0.98\textwidth]{external_pdf_files/#2.pdf}%
  \end{center}%
  \vspace{-0.5cm}%
  \caption{#3}%
\end{figure}\noindent
}

\ifpdf
\hypersetup{
  pdftitle={Asymptotic Log-Det Rank Minimization Via (Alternating) IRLS},
  pdfauthor={Sebastian Kr\"amer}
}
\fi

\begin{document}

\begingroup
% 
% REQUIRED
\begin{abstract}
    The affine rank minimization (ARM) problem is well known for both its applications and the fact that it is NP-hard. One of the most successful approaches, yet arguably underrepresented, is iteratively reweighted least squares (\IRLS{}), more specifically \IRLS{}-$0$. Despite comprehensive empirical evidence that it overall outperforms nuclear norm  minimization and related methods, it is still not understood to a satisfying degree. In particular, the significance of a slow decrease of the therein appearing regularization parameter denoted $\gamma$ poses interesting questions. While commonly equated to matrix recovery, we here consider the ARM independently. We investigate the particular structure and global convergence property behind the asymptotic minimization of the log-det objective function on which \IRLS{}-$0$ is based. We expand on local convergence theorems, now with an emphasis on the decline of $\gamma$, and provide representative examples as well as counterexamples such as a diverging \IRLS{}-$0$ sequence that clarify theoretical limits. We present a data sparse, alternating realization \AIRLS{}-$p$ (related to prior work under the name \SALSA{}) that, along with the rest of this work, serves as basis and introduction to the more general tensor setting. In conclusion, numerical sensitivity experiments are carried out that reconfirm the success of \IRLS{}-$0$ and demonstrate that in surprisingly many cases, a slower decay of $\gamma$ will yet lead to a solution of the ARM problem, up to the point that the exact theoretical phase transition for generic recoverability can be observed. Likewise, this suggests that non-convexity is less substantial and problematic for the log-det approach than it might initially appear.

\end{abstract}
\maketitle

\smallskip
\section{Introduction}\label{sec:intro}
Affine rank minimization (ARM) is, as the name suggests, the problem of finding a minimum rank matrix within an affine set, which may be provided in form of an underdetermined, linear equation. Given a linear operator $\mathcal{L}: \R^{n \times m} \rightarrow  \R^\ell$, for $n,m,\ell \in \N$ with $\ell < nm$, and a vector $y \in \mathrm{image}(\mathcal{L})$ referred to as measurements, we thus seek to find
\begin{align}\label{nosurrogate}
 X^\ast \in \underset{X \in \R^{n \times m}}{\mathrm{argmin}} \ \mathrm{rank}(X) \quad \mbox{subject to} \quad \mathcal{L}(X) = y.
\end{align}
Due to its numerous potential applications (for recent surveys, see \cite{DaRo16_AnO,NgKiSh19_Low})
, this setting is commonly equated to the matrix recovery problem, in which one desires to reconstruct a ground truth $X^{(\mathrm{true})}$ from measurements $y = \mathcal{L}(X^{(\mathrm{true})})$. Theoretical analysis of the latter however necessarily involves the question whether the found solution is in fact the sought recovery, which changes the perspective on the issue. We are here mainly interested in finding one minimizer $X^\ast$, regardless of whether it poses any sort of reconstruction. We leave the latter as the separate problem as which it can be viewed, though additionally attend to it in the numerical experiments.
\subsection{Approaches to ARM and matrix recovery}
A theoretical predecessor to the ARM problem is the closely related affine cardinality minimization (ACM), where (for in that case $\mathcal{L}: \R^n \rightarrow \R^\ell$) one seeks 
\begin{align}\label{acm}
 x^\ast \in \underset{x \in \R^n}{\mathrm{argmin}} \ \mathrm{card}(x) \quad \mbox{subject to} \quad \mathcal{L}(x) = y,
\end{align}
where $\mathrm{card}(x)$ is the number of non-zero coefficients of the vector. 
Like that problem, also ARM \cref{nosurrogate} is in general NP-hard \cite{Na95_Spa} and requires alternative, indirect approaches, where many of such are derived from the vector case. Analogous to the convex relaxation of the cardinality to the $\ell_1$-norm $\|x\|_1 = \sum_i |x_i|$, in the matrix version one replaces the rank function with the nuclear norm $\|X\|_\ast = \sum_{i = 1}^n \sigma_i(X)$. This approach has a long and well known history \cite{CaRe09_Exa,CaTa10_The,Gr11_Rec,Re11_ASi}, one of the earlier works being \cite{Fa02_Mat}.
Closely related to that relaxation is the so called null space property \cite{ReXuHa11_Nul,OyMoFaHa11_Asi}, the restricted isometry property \cite{ReFaPa10_Gua} as well as
the incoherence property \cite{CaTa10_The}. These have been major tools to derive lower, asymptotic bounds that guarantee successful recoveries with high probability. Optimization algorithm for nuclear norm minimization include methods (some of them very similar) known as soft thresholding \cite{CaCaSh10_ASi}, fixed-point continuation \cite{GoMa11_Con} and more \cite{LeBr10_Ato,HaMaLeZa15_Mat,KeMoOh10_Mat,KeMoOh10_Mat2}. Also iteratively reweighted least squares (\IRLS{}) has been applied \cite{FoRaWa11_Low}. Yet despite its theoretical rigorosity, mere $\ell_1$- or nuclear norm minimization seems to be outperformed by reweighted versions based on empirical results \cite{DaDeFoGu10_Ite,MoFa12_Ite,Ch08_Res,ChWo08_Ite} as we also observe in our numerical experiments in \cref{sec:numexp}.
Other methods, usually for matrix completion, rely on explicit rank estimates\footnote{At the latest in the setting of tensor recovery, a non invasively performed adaption of the rank poses an obstructive problem \cite{GrKr19_Sta}.} and usually optimize some data sparse model. This includes Riemannian optimization \cite{Va13_Low}, nonlinear successive overrelaxation \cite{XuZiWeZh12_Ana}, hard thresholding \cite{TaWe13_Nor} and others \cite{BaNoRe10_Onl,DaMiKe11_Sub}.
\subsection{Asymptotic minimization}\label{sec:asm}
We are not particularly concerned with assumptions that guarantee recoveries, but the general structure as well as convergence behavior of the herein laid out approach.
Essential therefor is the following family of objective functions, 
\begin{align} \label{matrixsur}
 f_\gamma(X) := \log \prod_{i = 1}^{n} (\sigma_i^2(X) + \gamma) = \log \det(X X^T + \gamma I), \quad \gamma \geq 0,
\end{align}
where $\sigma_i(X)$, $i = 1,\ldots,n$, are the singular values of $X \in \R^{n \times m}$, assuming without loss of generality $n \leq m$. When relevant, their domains are considered to be
\begin{align*}
 \mathrm{dom}(f_\gamma) = \mathcal{L}^{-1}(y)  = \{ X \in \R^{n \times m} \mid \mathcal{L}(X) = y\},
\end{align*}
for $\gamma > 0$, or, less importantly, $\mathrm{dom}(f_0) = \{ X \in \mathcal{L}^{-1}(y) \mid \mathrm{rank}(X) = n\}$.
This exact family has been previously considered for matrices in \cite{MoFa12_Ite} and closely related versions have been analyzed in \cite{Fa02_Mat,FaHiBo03_Log,MoFa10_Rew}.
These functions on the one hand follow as a certain least squares based generalization from vector reweighted $\ell_1$-minimization \cite{CaWaBo08_Enh}. Likewise, they can be derived as limit cases of the smoothed Schatten-$p$ functions \cite{MoFa12_Ite}, as similarly considered in the vector case \cite{DaDeFoGu10_Ite}, based on
\begin{align}\label{Schattenp}
 \lim_{p \rightarrow 0} \frac{S_{\gamma,p}(X) - {n}}{p} & = \frac{1}{2} f_{\gamma}(X) , \quad S_{\gamma,p}(X) := \sum_{i = 1}^{n} (\sigma_i^2(X) + \gamma)^{\frac{p}{2}}, \quad 0 < p \leq 1.
\end{align}
The one extreme case, $p = 1$, yields a smoothed version of the nuclear norm and is comprehensively covered in \cite{FoRaWa11_Low}. The opposing limit case, in which we are interested, denoted with $p = 0$, exhibits a distinct difference to all other choices of that parameter, on which we emphasize in the following remark. 
\begin{remark}\label{diffSf}
While minimizing $S_{\gamma,p}$ regardless of $0 < p \leq 1$, one ultimately searches for minimizers of $S_{0,p}$. Yet in general $f_0$ is not even to be minimized theoretically. In particular, every rank deficient matrix $X \in \mathcal{L}^{-1}(y)$ is already a minimizer of $\mathrm{det}(XX^T)|_{\mathcal{L}^{-1}(y)}$.
\end{remark}
However, we neither search for some specific $\gamma > 0$, 
but are interested in the asymptotic behavior of the minimizers for $\gamma \searrow 0$. 
We thus seek the following set.
\begin{definition} \label{Xastdef}
We define
 \begin{align*}
 \mathcal{X}^{\ast} := \{ X^\ast \mid \exists (X_\gamma)_{\gamma > 0} \subset \mathcal{L}^{-1}(y), \ X^\ast = \lim_{\gamma \searrow 0} X_\gamma, \ f_{\gamma}(X_\gamma) = \min_{X \in \mathcal{L}^{-1}(y)} f_{\gamma}(X)\}.
\end{align*}
\end{definition}
The importance of this process is also partially remarked on in \cite{MoFa12_Ite} and we underline in \cref{convergencelemma} that $\mathcal{X}^\ast$ indeed yields solutions to the ARM problem \cref{nosurrogate}. Naturally, as that set is defined via global minimizers of non convex functions, the challenge remains to find such and to theoretically or practically determine the chances of doing so.
The here considered $\{f_\gamma\}_{\gamma > 0}$ is not the only reasonable family, but it exhibits a particular structure (cf. \cref{sec:detexp}) and can be argued as natural choice (cf. \cref{sec:augmmatrix,rewidea}).
\subsection{Iteratively reweighted least squares (IRLS-p)}
The above mentioned maps $S_{\gamma,p}$ and $f_\gamma$, respectively, are of
particular practical interest since they can efficiently be minimized using so called iteratively reweighted least squares as covered in \cite{MoFa12_Ite,FoRaWa11_Low} as well as \cref{declinelemma}. Due to its dependency on a fixed weight strength parameter $p \in [0,1]$, the method is also abbreviated \IRLS{}-$p$.
When independently derived for ACM, the notion is to adaptively weight the entries of the vector iterate in order to balance out their magnitudes, thus bringing the process closer to the original cardinality. While there are different versions, one \cite{DaDeFoGu10_Ite} states to repeatedly solve
\begin{align}\label{vanillvectorairls}
 x^{(i)} = \underset{x \in \mathcal{L}^{-1}(y)}{\mathrm{argmin}} \ \|W_{\gamma^{(i-1)},x^{(i-1)}}^{1/2} x\|_2, \quad W_{\gamma,x} := \mathrm{diag}(x_1^2+\gamma,\ldots,x_n^2+\gamma)^{p/2-1},
\end{align}
for a monotonically decreasing sequence $\{\gamma^{(i)}\}_{i \geq 0} \subset \R_{>0}$.
In the limit $\gamma \searrow 0$, the $i$-th weight on the diagonal of $W^{1/2}_{\gamma,x}$ converges to $|x_i|^{p/2-1}$, or diverges to infinity, respectively. Thus, when viewing the non-zero part $x_{\neq 0} \in \R^{\mathrm{card(x)}}$, we have
\begin{align*}
 \|W^{1/2}_{0,x_{\neq 0}} x_{\neq 0}\|_2^2 = \sum_{i \in \mathrm{supp}(x)} |x_i|^p = \begin{cases}
                                                                                      \mathrm{card}(x), & p = 0 \\
                                                                                      \|x\|_1, & p = 1.
                                                                                     \end{cases}
\end{align*}
Our version for matrices \cite{MoFa12_Ite} is the analogous \cref{alg:irlsmr}, namely
\begin{align}\label{vanillairls}
 X^{(i)} = \underset{X \in \mathcal{L}^{-1}(y)}{\mathrm{argmin}} \ \|W_{\gamma^{(i-1)},X^{(i-1)}}^{1/2} X\|_F, \quad W_{\gamma,X} := (X X^T + \gamma I)^{p/2-1},
\end{align}
where $\|\cdot\|_F$ is the Frobenius norm.
The reasoning here is the same. For a sequence $X_\gamma \rightarrow \overline{X}$ with \textit{sufficiently fast} declining singular values $\sigma_i(X_\gamma)$, $i = 1,\ldots,n$, one has
\begin{align}\label{rewidea}
 \|W_{\gamma,X_\gamma}^{1/2} X_\gamma\|^2_F = \sum_{i = 1}^{n}
 \frac{\sigma_i(X_\gamma)^2}{(\sigma_i(X_\gamma)^2 + \gamma)^{1-p/2}} \underset{\gamma \searrow 0}{\longrightarrow} \sum_{i = 1}^{\mathrm{rank}(\overline{X})} \sigma_i(\overline{X})^p \overset{p = 0}{=} \mathrm{rank}(\overline{X}).
\end{align}
As above, one may also only view the non-zero singular values, whereby the term becomes continuous in $\gamma = 0$. Unfortunately, as it analogously holds true for the vector case, one can not in turn assume that the singular values of $X^{(i)}$ are in fact caused to decline fast enough to actually allow for \cref{rewidea}, so care has to be taken. What is here written as reweighting process \cref{vanillairls}, is also motivated as majorization-minimization algorithm for the vector case in \cite{CaWaBo08_Enh} and, similarly, as iterative linearization and minimization scheme for the matrix case in \cite{FaHiBo03_Log} as well as through entropy minimization in \cite{RaKr99_Ana}. So called log-thresholding has further appeared in \cite{MaAr14_Ite}.
While the to be minimized function is only convex for $p = 1$, our numerical test in \cref{sec:numexp} as well as the results in \cite{MoFa12_Ite,Ch08_Res,ChWo08_Ite} do suggest that $p = 0$ (representing $f_\gamma$) nevertheless seems to be the overall best choice. It is worth noting that given $W_{\gamma,X}$, it is possible to reconstruct $X X^T$, analogously to the vector case. So the iterate is tightly bound by the weights.
\subsection{Data sparse optimization}
Naturally, besides the issue of the theoretical convergence of \IRLS{}, also the computational complexity and possibly required relaxations are of importance. In \cite{HaMaLeZa15_Mat}, a gradient based projection method is considered for the matrix completion problem, whereas in \cite{FoRaWa11_Low}, the Woodbury matrix inversion lemma is utilized to more efficiently solve the reweighted least squares problems for separable operators. The authors of \cite{GiRoKo19_Alt} in turn analyze the factorized optimization of low rank matrices based on a tight upper bound for Schatten-$p$ functions. In order to realize the presented \IRLS{} method with low computational complexity, we consider an alternating version also based on the low rank matrix decomposition, which however directly minimizes the relaxated objective function $f_\gamma$.
\subsection{Contributions and organization of this paper}
The novel aspects of this paper are organized as follows:
\begin{itemize}
\item We consider affine rank minimization independently of recovery properties and remark on the possible degeneracy behind the problem in \cref{singularlemma}. 
\item  In \cref{sec:globalstructure}, we investigate the particular structure behind the asymptotic minimization approach (for $p = 0$) and prove global convergence by means of a more general, nested minimization scheme in \cref{convergencelemma}.
\item In \cref{sec:irls}, we expand on convergence statements about \IRLS{}-$0$ based on the insight from \cite{FoRaWa11_Low,MoFa12_Ite}, where here also \textit{complementary} weights are considered (see \cref{sec:mirror}). The decisive role of the sequence $\{\gamma^{(i)}\}_{i \geq 0}$, in particular both the necessity and sufficiency of a sufficiently slow decay, is illustrated theoretically and by representative, comprehensive examples in \cref{sec:algasp}. In \cref{divergencelemma}, it is shown that the iterates $X^{(i)}$ (see  \cref{vanillairls}) can in fact diverge, though this seems neglectable in practice.
\item In \cref{reloptmatrix}, we present an alternating \IRLS{}-$p$ method (\AIRLS{}-$p$, \cref{alg:datasparsematrix}) with low computational complexity, based on according switching between complementary weights (cf. \cref{sec:mirror}).
\item In \cref{sec:numexp}, numerical experiments, we reconfirm that the choice of the weight parameter $p = 0$ appears overall optimal.
We proceed sensibility tests with respect to each a constant rate of decline of $\gamma^{(i)} := \nu \gamma^{(i-1)}$, which expand on the results in \cite{MoFa12_Ite} and reveal that for some matrix problems, an excessively
slow decay may yet yield the desired solution.
We further demonstrate in \cref{exp2} that the relaxation of the affine constraint to a weighted penalty term appears to only marginally reduce the quality of minimizers, while the data sparse \AIRLS{}-$0$ in fact seems to improve upon the results.
\end{itemize}
\subsection{Degenerate ARM problems}\label{sec:degenerate}
As indicated above, we are not particularly interested in specific assumptions towards the ARM problem that guarantee some form of success, so that this remains subject to future research.
We however take note of the following issue since it is also related to the general convergence behavior of the approach considered in this work.
\begin{proposition}\label{singularlemma}
There exists a linear operator $\mathcal{L}$ and $y \in \mathrm{image}(\mathcal{L})$, such that 
there exists a sequence $\{X^{(k)}\}_{k = 1}^\infty \subset \mathcal{L}^{-1}(y)$ with
   \begin{align}
      \lim_{k \rightarrow \infty} \sigma_r(X^{(k)}) = 0, \quad \mbox{even though} \quad r = \min_{X \in \mathcal{L}^{-1}(y)} \rank(X). \label{singularproblem}
   \end{align}
\end{proposition}
\begin{proof}
 See \cref{divexample}.
\end{proof}
The sequence $X^{(k)}$ necessarily diverges. By \cref{singularproblem}, it directly follows that for the sequence of best rank $r-1$ approximations $X^{(k)}_{r-1}$, it holds true that
$\|X^{(k)} - X^{(k)}_{r-1}\| \rightarrow 0$ as well as $\|\mathcal{L}(X^{(k)}_{r-1}) - y\|_F \rightarrow 0$.
Note that it is trivial to construct problems where such sequences do not exist, and that it might be the generic case. In particular, at least despite reasonable effort, we could not find an instance of a sample based (or matrix completion) problem (cf. \cref{samplop}) that
allows for \cref{singularproblem} to happen. We provide the proof of \cref{singularlemma} in form of the following example.
\begin{example}\label{divexample}
 Let $\mathcal{L}: \R^{2 \times 3} \rightarrow \R^4$ be a linear operator and $y \in \R^4$ such that
 \begin{align*}
  \mathcal{L}^{-1}(y) = \{ X(a,b) \mid a,b \in \R \}, \quad X(a,b) := \begin{pmatrix}
                                                                       a & 1 & a + 1 \\
                                                                       b + 1 & b & b + 1
                                                                      \end{pmatrix}.
 \end{align*}
 Assume now that $\gamma = 0$. The function $f_{0}: \mathcal{L}^{-1}(y) \rightarrow \R$ (that is, for $\gamma = 0$\footnote{The value $\gamma = 0$ is for once legitimate as here $\mathcal{L}^{-1}(y)$ does not contain rank deficient matrices.}) has two stationary points. The (only) global minimum is at 
 \begin{align*}
  X = X(-1,-2/3) = \begin{pmatrix}
                                                                       -1 & 1 & 0 \\
                                                                       \frac{1}{3} & -\frac{2}{3} & \frac{1}{3}
                                                                      \end{pmatrix}, \quad \det(X X^T) = \frac{1}{3},
 \end{align*}
 and the only local maximum is at
 \begin{align*}
   X = X(1,0) = \begin{pmatrix}
                                                                       1 & 1 & 2 \\
                                                                       1 & 0 & 1
                                                                      \end{pmatrix}, \quad \det(X X^T) = 3.
 \end{align*}
 Thus, 
%  \begin{align}
 $r = 2 = \min_{X \in \mathcal{L}^{-1}(y)} \rank(X)$.
% \end{align}
However, for $(a_k,b_k) = (k,1/k)$ we have
\begin{align*}
 X^{(k)} = X(a_k,b_k) = \begin{pmatrix}
                                                                       k & 1 & k+1 \\
                                                                       1 + \frac{1}{k} & \frac{1}{k} & 1 + \frac{1}{k}
                                                                      \end{pmatrix}, \quad \det(X^{(k)} (X^{(k)})^T) = \frac{k^2 + 2k + 2}{k^2} \searrow 1.
\end{align*}
As clearly $\sigma_1(X^{(k)}) \rightarrow \infty$, it must at the same time hold true that
$\sigma_2(X^{(k)}) \rightarrow 0$.
So while there is no rank $1$ solution to $\mathcal{L}(X) = y$, there exists a sequence of matrices $X^{(k)}$ that
can arbitrarily closely be approximated by rank $1$ matrices $X_1^{(k)}$ for which further $\mathcal{L}(X_1^{(k)}) \rightarrow y$.
 In more detail, the points 
 \begin{align*}
  \{(a(b),b)\}_{b \neq 0} \subset \R^2, \quad a(b) := (3b+4+\sqrt{9 b^2+4})/(6b),
 \end{align*}
describe two disconnected valleys. The part $b < 0$ contains the global
 minimum of $f_0$, while $\lim_{b \searrow 0} a(b)\cdot b = 1$ and $\det(X(a(b),b) X(a(b),a)^T)$ is monotonically decreasing with limit $1$. For $b > 0$ small enough,
 there is thus no continuous path with monotonically decreasing function values from $(a(b),b)$ to any stationary point.
 \end{example}
The ARM problem \cref{divexample} also yields an instance of a divergent sequence for the \IRLS{} algorithm, as we discuss in detail in \cref{divergencelemma}. 
\section{Underlying structure and global behavior}\label{sec:globalstructure}
In this section, we investigate the structure behind the ARM problem and the behavior of the approach suggested by \cref{Xastdef}. Both affine cardinality \cref{acm} and rank minimization \cref{nosurrogate} can be attributed to a more general problem to find
\begin{align}\label{Xastgeneral}
 v^\ast \in \underset{v \in \mathcal{L}^{-1}(y)}{\mathrm{argmin}}\ \mathcal{C}_{\mathcal{V}}(v), \quad \mathcal{C}_{\mathcal{V}}(v) := \min_{V \in \mathcal{V}:\ v \in V} \mathrm{dim}(V),
\end{align}
where $\mathcal{V}$ is a family of varieties.
In the vector case, we simply have
\begin{align}\label{vectorV1}
 \mathcal{V}_1 = \{ V_S \mid S \subseteq \{1,\ldots,n\} \}, \quad V_S := \{x \in \R^n \mid \mathrm{supp}(x) \subseteq S\},
\end{align}
with $\mathrm{dim(V_S)} = |S|$, 
whereas the ARM problem is given by
\begin{align}
  \mathcal{V}_2 = \{ V_{\leq r} \mid r \in \N_0\}, \quad V_{\leq r}  := \{X \in \R^{n \times m}\mid \mathrm{rank}(X) \leq r\},
  \label{matrixvarieties}
\end{align}
for which one has $\mathrm{dim}(V_{\leq r}) = (n+m)r - r^2$.
Cardinality as opposed to rank minimization distinguish each other through the fact that $ \widetilde{V} \subsetneq V \Leftarrow \mathrm{dim}(\widetilde{V}) < \mathrm{dim}(V)$, $V \in \mathcal{V}$, holds true only in the matrix case.
As described in the following \cref{sec:detexp}, the log-det function $f_\gamma$ can be expanded as Taylor series in $\gamma$ into polynomials that reflect the above mentioned varieties (the same is possible for the vector version). We use that determinant expansion to prove the convergence result \cref{convergencelemma} essentially only relying on \cref{Xastgeneral}. A less general, but more direct proof is additionally given in \cref{sec:minorproofs}.
\subsection{Determinant expansion}\label{sec:detexp}
We define $\mathcal{P}_s([k]) := \{ I \subseteq \{1,\ldots,k\} \mid |I| = s\}$.
As indicated above, the function $f_\gamma$ can be expanded as Taylor series in $\gamma$.
\begin{proposition}
Let $X \in \R^{n \times m}$, $n \leq m$, and $\gamma \geq 0$. Then
\begin{align*}
 \det(X X^T + \gamma I) & = \gamma^{n} + \sum_{k = 1}^{n} \gamma^{n-k} \cdot \mathrm{det}^2_k(X), 
\end{align*}
where 
\begin{align*}
  \mathrm{det}^2_k(X) & := \sum_{I \in \mathcal{P}_k([n])} \sum_{J \in \mathcal{P}_k([m]) } \det(X_{I,J})^2,
\end{align*}
for $X_{I,J} := \{X_{i,j}\}_{i \in I,j \in J} \in \R^{|I| \times |J|}$.
\end{proposition}
\begin{proof}
 It is well known that $\det(A+\gamma I)$ is a polynomial in $\gamma$, that is
 \begin{align*}
  \det(X X^T + \gamma I) = \gamma^{n} + \sum_{k = 1}^{n} \gamma^{n-k} \cdot \sum_{J \in \mathcal{P}_k([n])} \det((X X^T)_{J,J}).
 \end{align*}
By the Cauchy-Binet formula, the latter term can be written as
\begin{align*}
 \det((X X^T)_{I,I}) = \sum_{J \in \mathcal{P}_k([n])} \det(X_{I,J}) \det((X^T)_{J,I}),
\end{align*}
whereas $\det(X_{I,J}) \det((X^T)_{J,I}) = \det(X_{I,J})^2$.
\end{proof}
As indicated earlier, the $k$-th summand determines membership towards the variety $V_{\leq k-1}$ defined in \cref{matrixvarieties}, since $\mathrm{det}^2_k(X) = 0  \Leftrightarrow  \rank(X) < k  \Leftrightarrow  X \in V_{\leq k-1}$. This well known equivalence is also proven through the following corollary.
\begin{corollary}\label{detktosingval}
 For each $k = 1,\ldots,m$, it is $\mathrm{det}^2_k(X) = \sum_{I \in \mathcal{P}_k([n])} \prod_{i \in I} \sigma_i(X)^2$. Thus, if $X$ is rank $r$, then $\mathrm{det}^2_r(X) = \prod_{i = 1}^r \sigma_i(X)^2$.
\end{corollary}
\begin{proof}
The result follows by simple comparison of polynomial degrees.
\end{proof}
\Cref{detktosingval} includes the special cases $ \sum_{i = 1}^{n} \sigma_i(X)^2 = \|X\|_F^2$ for $k = 1$,
 as well as $\prod_{i = 1}^{n} \sigma_i(X) = \det(X)$ for $k = n = m$. In particular, we have
 \begin{align}
  \|X\|_F^2 \leq \gamma^{1-{n}} \det(X X^T + \gamma I). \label{Frobeniusbound}
 \end{align}
Further, if $\rank(X) = r$, then for any low rank decomposition $X = Z Y$, $Y \in \R^{n \times r}$, $Z \in \R^{r \times m}$
it holds true that $\mathrm{det}_r^2(X) = \mathrm{det}_r^2(Y) \cdot \mathrm{det}_r^2(Z)$.
\subsection{Nested minimization}
In this section, we reduce the given framework to a more general setting which we denote as nested minimization scheme. For continuous functions $g_k: D \rightarrow \R_{\geq 0}$, where $D$ is a finite dimensional, affine set, let 
\begin{align*}
 G_\gamma(a) := \sum_{k = 1}^n \gamma^{n-k} g_k(a), \quad \gamma > 0, \ a \in D,
\end{align*}
as well as, as generalization of \cref{Xastdef},
\begin{align*}
 A^{\ast} := \{ a^\ast \in D \mid \exists (a_\gamma)_{\gamma > 0} \subset D, \ a^\ast = \lim_{\gamma \searrow 0} a_\gamma, \ G_\gamma(a_\gamma) = \min_{a \in D} G_\gamma(a)\}.
\end{align*}
 We further recursively define the sets $A_{n+1} := D$ and 
\begin{align*}
 A_k := \{ a \in A_{k+1} \mid g_k(a) =  \min_{b \in A_{k+1}} g_k(b) \}, \quad k = 1,\ldots,n.
\end{align*}
The assumptions \cref{assnested} in the following \cref{nestedminimization} are stronger than generally necessary, 
and we rather expect $A^\ast = A_1$, but in our ARM context it is not restrictive.
\begin{lemma} \label{nestedminimization}
 Let $s \geq 1$. Assume that for $k = s+1,\ldots,n$
 \begin{align}\label{assnested}
  \underset{a \in D}{\mathrm{argmin}} \ g_k(a) \subset \underset{a \in D}{\mathrm{argmin}} \ g_{k+1}(a).
 \end{align}
Then 
$A^\ast \subset A_s$.
Further, for a sequence of minimizers $a_\gamma \rightarrow a^\ast$, we have
 $|g_k(a_\gamma) - g_k(a^\ast)| \in \mathcal{O}(\gamma^{k-s})$,
for $k = s+1,\ldots,n$. 
\end{lemma}
Note that if aboves assumptions were to hold for $s = 0$, then each point in $A^\ast$ would already a minimizer of $G_\gamma$ for all $\gamma > 0$.
\begin{proof} See \cref{sec:minorproofs}.
\end{proof}
Contrary to the asymptotic behavior, larger values of $\gamma$ cause $G_\gamma$ to exhibit fewer local minima
if the functions $g_k$ become \textit{more convex} for smaller $k$. While this tendency 
is in principle promoted by \cref{assnested}, it is indeed observable for $G_\gamma = f_\gamma - \gamma^n$.
\subsection{Convergence of (global) minimizers}
We can now apply \cref{nestedminimization} with respect to $\mathcal{X}^\ast$ as in \cref{Xastdef} using the results in \cref{sec:detexp}.
\begin{theorem} \label{convergencelemma}
Let
$r = \min_{X \in \mathcal{L}^{-1}(y)} \rank(X)$.
 Then for any convergent sequence $\{X_\gamma\}_{\gamma > 0}$ of (global) minimizers of $f_\gamma(X)$ subject to $\mathcal{L}(X) = y$, we have 
 \begin{align}\label{globminconv}
  \lim_{\gamma \searrow 0} X_\gamma \in \underset{X \in \mathcal{L}^{-1}(y), \ \rank(X) = r}{\mathrm{argmin}}\ \prod_{i = 1}^r \sigma_i(X) \subset V_{\leq r},
 \end{align}
with 
$\sigma_{r+1}(X_\gamma)^2 \in \mathcal{O}(\gamma)$. % \label{sigmajbound}
If there is a unique rank $r$ minimizer $X_r$, then $X_\gamma \rightarrow X_r$.
\end{theorem}
As noted above, \cref{sec:visres} also contains a direct proof of \cref{convergencelemma} that is independent of \cref{nestedminimization}.
\begin{proof}
Let $D := \mathcal{L}^{-1}(y) \subset \R^{n \times m}$. The functions $g_k(X) := \mathrm{det}_k^2(X)|_{\mathcal{L}^{-1}(y)}$ fulfill the nestedness condition \cref{assnested} for $s = r$ (but generally not $s = r-1$),
whereas $A_{k} = \mathcal{L}^{-1}(y) \cap V_{\leq k-1} = g_k^{-1}(0)$, $k = r+1,\ldots,n$.
For $X_r \in A_{r+1}$, we further have $g_r(X_r) = \prod_{i=1}^r \sigma_i(X_r)^2$. The remaining bound then follows by $C \sigma_{r+1}(X_\gamma)^2 \leq \prod_{i=1}^{r+1} \sigma_i(X_\gamma)^2 \leq g_{r+1}(X_\gamma) = |g_{r+1}(X_\gamma) - g_{r+1}(X_r)| \in \mathcal{O}( \gamma^{(r+1)-r})$ for some $C > 0$.
\end{proof}
With the presumably possible weakening of the assumption \cref{assnested}, we conjecture that the limit of $X_\gamma$ will (almost always) additionally minimize $g_{k}(X) = \mathrm{det}_k^2(X)$ subject to each priorly admissible set, for $k = r-1,r-2,\ldots,1$ (naturally, this becomes trivial once it becomes uniquely determined). 
\section{Log-det iteratively reweighted least squares (IRLS)}\label{sec:irls}
Minimizing the function $f_\gamma$ or finding its extremal points directly is likely not practicable. The strategy of \IRLS{} instead provides remedy by 
introducing an artificial variable in form of a weight matrix. 
\subsection{Minimization of an augmented function}\label{sec:augmmatrix}
We will hint at how to reversely derive the following function, but for now as in \cite{MoFa12_Ite}\footnote{There is no inherent mathematical difference (cf. \cref{sec:mirror}) between using $X^T X$ as in \cite{MoFa12_Ite} or $X X^T$. We however use the latter due to its proximity to \cite{FoRaWa11_Low}.} we define
\begin{align}
 J_{\gamma}(X,W) :=&\ \mathrm{trace}(W(X X^T + \gamma I)) - \log \det(W) - n \label{matrixJ} \\
 =&\ \|W^{1/2} X \|_F^2 + \gamma \|W^{1/2}\|_F^2 - \log \det(W) - n \nonumber,
\end{align}
where $W \in \R^{n \times n}$ ranges over all symmetric positive definite matrices, denoted with $W = W^T \succ 0$.
The matrix $W$ is also called weight matrix, the reason of which will become apparent in this section. We transfer the concepts from \cite{FoRaWa11_Low} as we will need it in the following for the a little different case we are given here. Most results essentially appear in \cite{MoFa12_Ite}, but we do use the methodology from \cite{FoRaWa11_Low}.
\begin{lemma}[cf. \cite{FoRaWa11_Low,MoFa12_Ite}]\label{partialW}%
The Fr\'echet derivative of $J_{\gamma}$ with respect to $W$ is
 \begin{align*}
  \frac{\partial}{\partial W}\ J_\gamma(X,W) = X X^T + \gamma I - W^{-1}.
 \end{align*}
\end{lemma}
\begin{proof}
 Since $\mathrm{trace}(W(X X^T + \gamma I))$ is linear in $W$, it follows 
 \begin{align*}
  \frac{\partial}{\partial W}\ \mathrm{trace}(W(X X^T + \gamma I)) = X X^T + \gamma I.
 \end{align*}
Further, as
$\log \det(W) = \sum_{i = 1}^n \log \lambda_i(W)$
is a function that depends only on the eigenvalues $\lambda_i(W)$ of $W =: U \Lambda U^T$, it follows as described in \cite{LeSe05_Non,FoRaWa11_Low}
that
 \begin{align*}
  \frac{\partial}{\partial W} \log \det(W) = U \diag(\lambda_1(W)^{-1},\ldots,\lambda_n(W)^{-1}) U^T = W^{-1}.
 \end{align*}
Naturally, the constant $n$ vanishes.
\end{proof}
The function $J_\gamma$ hence has a unique minimizer in $W$.
\begin{corollary}[cf. \cite{FoRaWa11_Low,MoFa12_Ite}]\label{W_gX}%
The minimizer of $J_\gamma$ in $W$ is given by
 \begin{align*}
 W_{\gamma,X} :=&\ \underset{W = W^T \succ 0}{\mathrm{argmin}} J_\gamma(X,W) = (X X^T + \gamma I)^{-1} \\
 =&\ U \diag((\sigma_1(X)^2 + \gamma)^{-1},\ldots,(\sigma_n(X)^2 + \gamma)^{-1}) U^T,
\end{align*}
for the SVD $X = U \Sigma V^T$.
\end{corollary}
Given the nature of the elementary functions $\log(x)$ and $\frac{1}{x}$, we have that $W_{\gamma,X}$ remains bounded if and only if $f_\gamma(X)$ remains bounded from below. 
We obtain the following important assertion, which connects the functions $J_\gamma$ and $f_\gamma$.
\begin{lemma}[essentially \cite{FoRaWa11_Low,MoFa12_Ite}]\label{Jtof}
 For the minimizer $W_{\gamma,X}$ of $J_\gamma(X,W)$ subject to $W = W^T \succ 0$, it holds 
$f_\gamma(X) = J_\gamma(X,W_{\gamma,X})$.
\end{lemma}
\begin{proof}
This follows from the previous discussion as
 \begin{align*}
 J_\gamma(X,W_{\gamma,X}) = - \log \det ((X X^T + \gamma I)^{-1}) = \log \det (X X^T + \gamma I).
\end{align*}
\end{proof}
Instead of minimizing the lefthand function $f_\gamma(X)$, one thus turns to the alternating minimization of $J_\gamma(X,W)$.
\begin{corollary}
 For $\mathcal{X}^\ast$ as in \cref{Xastdef}, it holds
  \begin{align*}
 \mathcal{X}^{\ast} = \{ X^\ast \mid \exists &(X_\gamma,W_\gamma)_{\gamma > 0}: \ X^\ast = \lim_{\gamma \searrow 0} X_\gamma, \
 J_\gamma(X_\gamma,W_\gamma) = \min_{\begin{array}{c} \scriptstyle  X \in \mathcal{L}^{-1}(y) \\ \scriptstyle  W = W^T \succ 0 \end{array}} J_\gamma(X,W)\}.
\end{align*}
\end{corollary}
A simple least squares problem gives the minimizer in $X$ as
\begin{align}\label{Xsimplels}
 \underset{X \in \mathcal{L}^{-1}(y)}{\mathrm{argmin}} J_\gamma(X,W) = \underset{X \in \mathcal{L}^{-1}(y)}{\mathrm{argmin}} \|W^{1/2} X\|_F^2,
\end{align}
resulting in the following update formula.
\begin{lemma}[cf. \cite{FoRaWa11_Low,MoFa12_Ite}]\label{X_W}
Let $W = W^T \succ 0$. Then
\begin{align*}
 X_W := \underset{X \in \mathcal{L}^{-1}(y)}{\mathrm{argmin}} J_\gamma(X,W) = 
 \mathcal{W}^{-1} \circ \mathcal{L}^\ast \circ (\mathcal{L} \circ \mathcal{W}^{-1} \circ \mathcal{L}^\ast)^{-1}(y),
\end{align*} 
for
$\mathcal{W}^{-1}(X) := W^{-1} X$,
where $\mathcal{L}^\ast$ is the adjoint of $\mathcal{L}$.
\end{lemma}
\begin{proof}
 Follows by \cref{minimization} applied to \cref{Xsimplels}.
\end{proof}
\begin{corollary}\label{rationalupdate}
 Each entry of the update $X_{W_{\gamma,X}}$ is a rational function in $\gamma$ and the entries of $X$.
\end{corollary}
For relatively small $\gamma$, a more stable, although computationally more demanding update formula is provided by \cref{minimization} through
\begin{align}\label{kernelbased}
 X_W = X_0 - \mathcal{K} \circ (\mathcal{K}^\ast \circ \mathcal{W} \circ \mathcal{K})^{-1} \circ \mathcal{K}^\ast \circ \mathcal{W}(X_0),
\end{align}
where $\mathcal{K}: \R^{n m - \ell} \rightarrow \R^{n \times m}$ is a kernel representation of $\mathcal{L}$, thus $\mathrm{image}(\mathcal{K}) = \mathrm{kernel}(\mathcal{L})$. $X_0$ may be any one solution to $\mathcal{L}(X_0) = y$, for instance the first or previous iterate. In the following, let $\perp$ be orthogonality with respect to the Frobenius scalar product.
 As for any matrices $X$ and $W \succ 0$ the following equivalences hold true,
\begin{align*}
   W X \perp \mathrm{kernel}(\mathcal{L}) \quad \Leftrightarrow \quad W X  \in \mathrm{range}(\mathcal{L}^\ast) \quad \Leftrightarrow \quad X \in \mathrm{range}(\mathcal{W}^{-1} \circ \mathcal{L}^\ast),
\end{align*}
the previous \cref{X_W} also provides that
\begin{align}
 W X_W \perp \mathrm{kernel}(\mathcal{L}). \label{eq:X_WWortho}
\end{align}
Conversely, $X_W$ is the unique solution to $\frac{\partial}{\partial X} J_\gamma(X,W) = W X \perp \mathrm{kernel}(\mathcal{L})$ subject to $\mathcal{L}(X) = y$,
which provides an alternative proof.\\\\
The weight matrix $W_{\gamma,X}$ is, as indicated in \cref{rewidea}, in the following sense an optimal choice. Since for $r = \mathrm{rank}(X)$ we have
\begin{align}
 \big\|W_{\gamma,X}^{1/2} X \big\|_F^2 & = \big\|\diag((\sigma_1(X)^2+\gamma)^{-1/2},\ldots,(\sigma_r(X)^2+\gamma)^{-1/2},\gamma^{-1/2},\ldots,\gamma^{-1/2}) \Sigma\, \big\|^2_F\nonumber \\
 & = \sum_{i = 1}^r \sigma_i(X)^2 \cdot (\sigma_i(X)^2+\gamma)^{-1} \underset{\gamma \searrow 0}{\longrightarrow} \mathrm{rank}(X). \label{eq:rankprop}
\end{align}
Also the stationary points of $f_\gamma$ and $J_\gamma$ are directly related, as follows.
\begin{theorem}\label{stabilizerlemma}
 We have
 \begin{align*}
  \nabla_X f_\gamma(X) = \nabla_X J_\gamma(X,W)|_{W = W_{\gamma,X}} = W_{\gamma,X} X . 
 \end{align*}
 Thus $X$ is a stationary point of $f_\gamma$ if and only if $X = X_{W}$ for $W = W_{\gamma,X}$,
 which means that $(X,W_{\gamma,X})$ is a stationary point of $J_\gamma$.
\end{theorem}
\begin{proof}
 The gradient identity follows by chain differentiation as 
 \begin{align*}
  \nabla_X f_\gamma(X) = \nabla_X J_\gamma(X,W_{\gamma,X})
 \end{align*}
and $\nabla_W J_\gamma(X,W)|_{W = W_{\gamma,X}} = 0$ for all $X \in \mathcal{L}^{-1}(y)$. We then have
\begin{align*}
 \nabla_X f_\gamma(X) \perp \mathrm{kernel}(\mathcal{L}) \quad & \overset{\hphantom{\cref{eq:X_WWortho}}}{\Leftrightarrow} \quad \nabla_X J_\gamma(X,W)|_{W = W_{\gamma,X}} \perp \mathrm{kernel}(\mathcal{L}) \\
 & \overset{\cref{eq:X_WWortho}}{\Leftrightarrow} \quad X = X_W, \ W = W_{\gamma,X}.
\end{align*}
Stationary points of $J_\gamma$ in turn are indeed those pairs $(X,W)$ for which $X = X_W$ and $W = W_{\gamma,X}$.
\end{proof}
The relations laid out in this section can vice versa be postulated and be used to derive the function $f_\gamma$ even without constructing $J_\gamma$. Central therein is the aim to represent the property \cref{eq:rankprop}.
As provided by the following, the limit case $\gamma \rightarrow \infty$ provides a uniquely determined starting value.
\begin{lemma}\label{gammainf}
 Independently of $X^{(0)} \in \mathcal{L}^{-1}(y)$, it holds
 \begin{align*}
  \lim_{\gamma \rightarrow \infty} \underset{X \in \mathcal{L}^{-1}(y)}{\mathrm{argmin}} f_\gamma(X) = \lim_{\gamma \rightarrow \infty} X_{W_{\gamma,X^{(0)}}}
  = \underset{X \in \mathcal{L}^{-1}(y)}{\mathrm{argmin}} \|X\|_F,
 \end{align*}
 where the first limit is possibly a set convergence.
\end{lemma}
\begin{proof}
 The second equality follows since $\measuredangle (W_{\gamma,X^{(0)}}, I_n) \rightarrow 0$.
 Let therefore $X^\infty = \mathrm{argmin} \|X\|_F$ subject to $\mathcal{L}(X) = y$. For $\gamma > 0$, let further $X \in \mathcal{L}^{-1}(y)$ with $f_\gamma(X) \leq f_\gamma(X^\infty)$.
 Then since $f_\gamma(X^\infty) \leq \gamma^n + \gamma^{n-1} \|X^\infty\|_F^2 + \gamma^{n-2} c$ (cf. \cref{sec:detexp}), for some fixed $c > 0$, it follows that $\|X\|_F^2 \leq \|X^\infty\|_F^2 + c \gamma^{-1}$. Thereby, since $X^\infty \perp \mathrm{kernel}(\mathcal{L})$, we have
 $\|X - X^\infty\|_F^2 = \|X\|_F^2 - \|X^\infty\|_F^2 \leq c \gamma^{-1}$. Any global minima of $f_\gamma$ must fulfill this bound, and with $\gamma \rightarrow \infty$ it follows $X = X^\infty$.
\end{proof}
\subsection{Complementary weights}\label{sec:mirror}
We have so far only considered the version $f_\gamma(X) = \log \det(X X^T + \gamma I)$, but
all statements in \cref{sec:augmmatrix} analogously hold true as well for the complementary\footnote{Denoted as such since the contraction of $X$ with itself is over the complementary modes.} versions (as used in \cite{MoFa12_Ite})
\begin{align*}
  f_\gamma^{(2)} :=&\ \log \det(X^T X + \gamma I), \\
  J^{(2)}_\gamma(X,W^{(2)}) :=&\ \mathrm{trace}(W^{(2)} (X^T X + \gamma I)) - \log \det(W^{(2)}) - m \\
  =&\ \| X (W^{(2)})^{1/2} \|_F^2 + \gamma \|(W^{(2)})^{1/2}\|_F^2 - \log \det(W^{(2)}) - m.
\end{align*}
Further, while simply 
 $f^{(2)}_\gamma(X) = \sum_{i = 1}^m \log(\sigma_i(X)^2 + \gamma) = f_\gamma(X) + \log( \gamma ) \cdot(m-n)$,
the updates in $X$ and $W^{(2)}$ corresponding to $J^{(2)}_\gamma(X,W^{(2)})$ are given by
\begin{align}
 X^{(2)}_{W^{(2)}} &:= \underset{X \in \mathcal{L}^{-1}(y)}{\mathrm{argmin}} J^{(2)}_\gamma(X,W^{(2)}) = \underset{X \in \mathcal{L}^{-1}(y)}{\mathrm{argmin}} \|X (W^{(2)})^{1/2} \|_F^2, \label{eq:X2_W} \\
 W^{(2)}_{\gamma,X} &:= \underset{W^{(2)} = (W^{(2)})^T \succ 0}{\mathrm{argmin}} J^{(2)}_\gamma(X,W^{(2)}) = (X^T X + \gamma I)^{-1}. \label{eq:W2_gX}
\end{align}
In the following, when appropriate, we also use $f^{(1)}_\gamma := f_\gamma$ and $J^{(1)}_\gamma := J_\gamma$.
While interchangeable as such, a combination of the two versions proves relevant in \cref{reloptmatrix}.
\subsection{Adjusted IRLS-p algorithm}
As indicated in \cref{sec:mirror}, there are two possible, in general different, 
updates to $X$ depending on the choice of weight, which we denote via a sequence $\{s_i\}_{i \geq 0} \subset \{1,2\}$. 
\Cref{alg:irlsmr} further depends on a weakly decreasing, 
countable sequence $\{\gamma_i\}_{i \geq 0} \subset \R_{> 0}$. 
Based thereon, it defines a sequence $\{(X^{(i)},W^{(s_i,i)})\}_{i \geq 0}$. With respect to \cref{gammainf}, choosing $X^{(0)} = X_I$ yields a canonical starting value, together with a sufficiently large $\gamma^{(0)}$.
 \begin{algorithm}
  \caption{(matrix) \IRLS{}-$0$}
  \begin{algorithmic}[1] \label{alg:irlsmr}
  \STATE set $X^{(0)}$, $\gamma^{(0)} > 0$  
  \FOR{$i = 1,2,\ldots$}
   \STATE{set $s_{i-1} \in \{1,2\}$ (cf. \cref{sec:mirror})}
  \STATE{$W^{(s_{i-1},i-1)}  := W^{(s_{i-1})}_{\gamma^{(i-1)},X^{(i-1)}}$ (\cref{W_gX,eq:W2_gX})}
  \STATE{$X^{(i)} := X^{(s_{i-1})}_{W^{(s_{i-1},i-1)}}$ (\cref{X_W,eq:X2_W})}
  \STATE{set $\gamma^{(i)} \leq \gamma^{(i-1)}$}
  \ENDFOR
  \end{algorithmic}
\end{algorithm}
While the choice $s_i = s_j$, $i,j \geq 0$, yields conventional, well working \IRLS{}-$0$, alternating between complementary weights $(s_{2i},s_{2i+1}) = (1,2)$, $i \geq 0$, become decisive for the data sparse algorithm presented in \cref{reloptmatrix}.
\paragraph{Possible divergence}
Although it seems neglectable in practice, \cref{divergencelemma} shows that in principle, it is possible for the sequence $X^{(i)}$ to diverge (at a glacial pace though) for $\gamma^{(i)} \rightarrow 0$.
The therein used problem setting is the same as in \cref{divexample}.
\paragraph{Asymptotic and global behavior regarding $\gamma \searrow 0$}
Not only does \cref{convergenceexample} demonstrate that $\gamma$ can not simply be set as $0$,
it proves that if $\gamma$ is decreased too fast, the function $f_\gamma(X)$ will become flat locally along rank deficient matrices
more quickly than its minimization proceeds. 
In that case, the iterate may converge, but to a point that is neither a stationary point of $f_\gamma(X)$ nor a limit of such for $\gamma \searrow 0$.
On a global scale in turn, $\gamma$ acts similar as a median filter on $\det(X X^T)$,
and in that sense (in the optimal case) smoothes out undesired local minima.
\paragraph{Controlling the decline of $\gamma$}
Both works \cite{DaDeFoGu10_Ite,FoRaWa11_Low} use what here translates to $\gamma^{(i)} = \alpha \sigma_{K+1}(X^{(i)})$ for some $\alpha \in (0,1]$ and a sufficiently large bound $K \in \N$ on the to be found rank. However, while this strategy may not only be of theoretical benefit for the minimization of the differently behaving $S_{\gamma,p}$, $0 < p \leq 1$ (see \cref{diffSf}), we found that it will frequently cause the iteration to stagnate or, in particular for $p = 0$, cause a too rapid decay of $\gamma$. 
We instead consider a fixed rate of decline, $\gamma^{(i)} = \nu \gamma^{(i-1)}$, $\nu \in (0,1)$, if not otherwise indicated.
\subsection{Local convergence and asymptotic, stationary points}\label{sec:algasp}
The following parts $(i)$ to $(iii)$ of \cref{declinelemma} are directly based on results and reasoning from \cite{FoRaWa11_Low,MoFa12_Ite}, but also takes switching between complementary weights into account. We further extend it with part $(iv)$ which particularly considers the rate of decline of $\gamma$. We define $\mathcal{S}^\ast_\gamma$ as the stationary points\footnote{The function $f^{(2)}_\gamma$ shows the same behavior as it only differs by a constant in $\gamma$.} of $f_\gamma|_{\mathcal{L}^{-1}(y)}$, $\gamma > 0$ (subject to their respective domains).
\begin{theorem} \label{declinelemma}
 Let $\{(X^{(i)})\}_{i \geq 0}$ be generated by \cref{alg:irlsmr} for $\{s_i\}_{i \in \N_0}$
 and the weakly decreasing sequence $\{\gamma_i\}_{i \geq 0} \subset \R_{\geq 0}$ and let $\gamma^\ast := \lim_{i \rightarrow \infty} \gamma^{(i)}$.
 \begin{enumerate}[label=(\roman*)]
  \item For each $i \in \N$ and both $s \in \{1,2\}$, it holds \begin{align}\label{eq:mondec}
	  f^{(s)}_{\gamma^{(i)}}(X^{(i)}) \leq f^{(s)}_{\gamma^{(i-1)}}(X^{(i-1)}). 
	\end{align}
 \item If $\gamma^\ast > 0$, then the sequences $X^{(i)}$ and $|f^{(s)}_{\gamma^{(i)}}(X^{(i)})|$, $s \in \{1,2\}$, remain bounded.
  \item If the sequences $X^{(i)}$ and $|f_{\gamma^{(i)}}(X^{(i)})|$ remain bounded, then
      \begin{align} \label{difftozero}
        \lim_{i \rightarrow \infty} \|X^{(i)} - X^{(i-1)}\|_F = 0
       \end{align}
 and each accumulation point of $X^{(i)}$ is in $\mathcal{S}^\ast_{\gamma^\ast}$.
 \item (See \cref{iveasy}) Let $\Theta \subset \R_{> 0}$ be an arbitrary, infinite, bounded set with its only accumulation point at $\inf(\Theta) = 0$, and let
 \begin{align*}
  \delta_i := \inf_{S \in \mathcal{S}_{\gamma^{(i)}}^\ast} \| X^{(i)} - S \|, \quad i \in \N.
 \end{align*}
 For an arbitrary, bounded sequence $A = \{\alpha_i\}_{i \in \N_0}$ with $\inf(A) > 0$ (e.g. $\alpha_i = 1$, $i \in \N_0$) and for $\gamma^{(0)} = \max(\Theta)$, we recursively define
 \begin{align*}
 \gamma^{(i+1)} = \begin{cases}
		   \theta_i & \mbox{ if } \alpha_i \delta_i < \theta_i \\
		    \gamma^{(i)} & \mbox{ otherwise }
                  \end{cases}, \quad \theta_i := \max \{ z \in \Theta \mid z < \gamma^{(i)} \}, \quad i \in \N_0.
\end{align*}
 Then $\lim_{i \rightarrow \infty} \delta_i = \gamma^\ast = 0$ and for at least one subsequence $\{X^{(i_\ell)}\}_{\ell \in \N}$, there exists a sequence of stationary points $\{S_{\ell}\}_{\ell \in \N}$, $S_{\ell} \in \mathcal{S}^\ast_{\gamma^{(i_\ell)}}$, with $\|S_\ell - X^{(i_\ell)}\| \rightarrow 0$.
 \end{enumerate}
\end{theorem}%
\begin{remark}\label{iveasy}
Part $(iv)$ of \cref{declinelemma} can roughly be phrased as the following. If the sequence $\{\gamma^{(i)}\}_{i \in \N}$ is decreased to $\gamma^\ast = 0$ slowly enough,
then $X^{(i)}$ can only converge to a limit of stationary points of $f_\gamma|_{\mathcal{L}^{-1}(y)}$ for $\gamma \searrow 0$.
The contrary case of too fast decline is covered in \cref{para:toofastdecl}.
\end{remark}
\begin{proof}
$(i)$: For $h := (n,m) \in \N^2$ and independent of $s \in \{1,2\}$, we have
 \begin{align*}
 f^{(s)}_{\gamma^{(i)}}(X^{(i)}) & \overset{(a)}{=} f^{(s_{i})}_{\gamma^{(i)}}(X^{(i)}) + \gamma^{(i)} (h_s - h_{s_{i}})  \\
 & \overset{(b)}{=} J^{(s_{i})}_{\gamma^{(i)}}(X^{(i)},W^{(s_{i},i)}) + \gamma^{(i)} (h_s - h_{s_{i}}) \\
 & \overset{(c)}{\geq} J^{(s_{i})}_{\gamma^{(i)}}(X^{(i+1)},W^{(s_{i},i)}) + \gamma^{(i)} (h_s - h_{s_{i}}) \\
 & \overset{(d)}{\geq} J^{(s_{i})}_{\gamma^{(i)}}(X^{(i+1)},W^{(s_{i})}_{\gamma^{(i)},X^{(i+1)}}) + \gamma^{(i)} (h_s - h_{s_{i}}) \\
 & \overset{(e)}{=} f^{(s_{i})}_{\gamma^{(i)}}(X^{(i+1)}) + \gamma^{(i)} (h_s - h_{s_{i}}) \overset{(f)}{=} f^{(s)}_{\gamma^{(i)}}(X^{(i+1)}) \overset{(g)}{\geq} f^{(s)}_{\gamma^{(i+1)}}(X^{(i+1)}).
 \end{align*}
  The steps $(a)$ to $(g)$ are provided by: $(a)$ \cref{sec:mirror}, $(b)$ \cref{Jtof}, $(c)$ $X^{(i+1)} = X^{(s_{i})}_{W^{(s_{i},i)}}$ is optimum in $X$ (\cref{X_W}), $(d)$ $W^{(s_{i})}_{\gamma^{(i)},X^{(i+1)}}$ is the respective optimum in $W$ (\cref{W_gX}), $(e)$ \cref{Jtof}, $(f)$ \cref{sec:mirror} and 
  $(g)$ $\frac{\partial}{\partial \gamma} f^{(s)}_\gamma(X) \geq 0$, $s \in \{1,2\}$, for all $X$. 
  In contrast to the usual argumentation, we here require the intermediate, practically redundant step $(d)$. \\
 $(ii)$: Since (cf. \cref{Frobeniusbound})
    $\gamma^{n  - 1} \|X\|_F^{2} \leq \prod_{i=1}^{n} (\sigma_i(X)^2 + \gamma) = \exp(f_{\gamma}(X))$,
 it follows due to $(i)$ (since $f^{(1)}_\gamma = f_\gamma$) that 
$\|X^{(i)}\|_F^{2} \leq (\gamma^{(i)})^{1-n} \exp(f_{\gamma^{(1)}}(X^{(1)}))$.
 As $\gamma^{(i)}$ does not converge to zero, the sequence $X^{(i)}$ remains bounded.\\
 $(iii/1)$: 
 For $s = s_i$ (and thus $h_s - h_{s_i} = 0$), the steps $(d)$ to $(g)$ in $(i)$ provide that
 $J^{(s_i)}_{\gamma^{(i)}}(X^{(i+1)},W^{(s_i,i)}) \geq f^{(s_i)}_{\gamma^{(i+1)}}(X^{(i+1)})$.
 With $\langle X, W, X \rangle_1 := \mathrm{trace}(W X X^T)$ and $\langle X, W, X \rangle_2 := \mathrm{trace}(X^T X W)$, it then follows that
 \begin{align*}
 &\  f^{(s_i)}_{\gamma^{(i)}}(X^{(i)}) - f^{(s_i)}_{\gamma^{(i+1)}}(X^{(i+1)}) \\
  \geq &\  J^{(s_i)}_{\gamma^{(i)}}(X^{(i)},W^{(s_i,i)}) - J^{(s_i)}_{\gamma^{(i)}}(X^{(i+1)},W^{(s_i,i)}) \\
  = &\ \langle X^{(i)}, W^{(s_i,i)}, X^{(i)} \rangle_{s_i} - \langle X^{(i+1)}, W^{(s_i,i)}, X^{(i+1)} \rangle_{s_i} \\
  = &\ \langle X^{(i)} - X^{(i+1)}, W^{(s_i,i)}, X^{(i)} + X^{(i+1)} \rangle_{s_i}.
 \end{align*}
 Since $X^{(i+1)} = X^{(s_i)}_{W^{(s_i,i)}}$ and $X^{(i)} - X^{(i+1)} \in \mathrm{kernel}(\mathcal{L})$, the optimality condition \cref{eq:X_WWortho} provides that $\langle X^{(i)} - X^{(i+1)}, W^{(s_i,i)}, X^{(i+1)} \rangle_{s_i} = 0$.
 We can thus conclude
   \begin{align*}
  \langle X^{(i)} - X^{(i+1)}, W^{(s_i,i)}, X^{(i)} + X^{(i+1)} \rangle_{s_i} & = \langle X^{(i)} - X^{(i+1)}, W^{(s_i,i)},X^{(i)} - X^{(i+1)} \rangle_{s_i} \\
  & \geq \| (X^{(i)} - X^{(i+1)}) \|^2_F\  \lambda_{\min}({W^{(s_i,i)}}).
 \end{align*}
 The lowest eigenvalue of the symmetric matrix can be bounded via
 \begin{align*}
  \lambda_{\min}({W^{(s_i,i)}}) 
  & = (\sigma_1({X^{(i)}})^2 + \gamma)^{-1} \geq  \ (\|X\|_F^2 + \gamma)^{-1}.
 \end{align*}
 Thereby, as $\|X\|_F^2$ remains bounded by assumption, there exists $c > 0$ such that
 \begin{align*}
  \| (X^{(i)} - X^{(i+1)}) \|^2_F\  \lambda_{\min}({W^{(s_i,i)}}) \geq c \, \| (X^{(i)} - X^{(i+1)}) \|^2_F.
 \end{align*}
 Summing over all $i = 1,\ldots,N$, we obtain
 \begin{align*}
  &\ c \sum_{i = 1}^N \| (X^{(i)} - X^{(i+1)}) \|^2_F  \leq \sum_{i = 1}^N f^{s_i}_{\gamma^{(i)}}(X^{(i)}) - f^{s_i}_{\gamma^{(i+1)}}(X^{(i+1)}) \\
   \overset{\cref{eq:mondec}}{\leq} &\  \sum_{i = 1}^N f^{(1)}_{\gamma^{(i)}}(X^{(i)}) - f^{(1)}_{\gamma^{(i+1)}}(X^{(i+1)}) + \sum_{i = 1}^N f^{(2)}_{\gamma^{(i)}}(X^{(i)}) - f^{(2)}_{\gamma^{(i+1)}}(X^{(i+1)}) \\
    =&\  f^{(1)}_{\gamma^{(1)}}(X^{(1)}) - f^{(1)}_{\gamma^{(N+1)}}(X^{(N+1)}) + f^{(2)}_{\gamma^{(1)}}(X^{(1)}) - f^{(2)}_{\gamma^{(N+1)}}(X^{(N+1)}).% \\
%    &\leq f_{\gamma^{(1)}}(X^{(1)}),
 \end{align*}
 As $f_{\gamma^{(N+1)}}(X^{(N+1)})$ (and thereby $f^{(2)}_{\gamma^{(N+1)}}(X^{(N+1)})$) remains bounded by assumption as well, the sum can not diverge, and it necessarily follows the to be shown $\| (X^{(i)} - X^{(i+1)}) \|_F \rightarrow 0$ for $i \rightarrow \infty$.\\
$(iii/2)$: For this part, it suffices to consider the initial version $f_\gamma = f^{(1)}_\gamma$,
wherefore we skip the index $(\cdot)^{(1)}$.
Let $X^{(i_\ell)}$ be a convergent subsequence of $X^{(i)}$ with limit point $X^\ast$.
In light of \cref{stabilizerlemma}, we need to show that $X^\ast = X^\ast_{W^\ast}$ for $W^{(\ast)} = W_{\gamma^\ast,X^\ast}$.
Due to $(iii/1)$ so far, we have $\lim_{\ell \rightarrow \infty} X^{(i_\ell+1)} = X^\ast$. As $W_{\gamma,X}$ depends continuously on $X$ as long as $f_\gamma(X)$ remains bounded (which may directly be implied by $\gamma^\ast > 0$), 
it follows that
\begin{align*}
 W^{(i_\ell)} = W_{\gamma^{(i_\ell)},X^{(i_\ell)}} \rightarrow_{i \rightarrow \infty} W_{\gamma^\ast,X^\ast} =: W^{(\ast)}.
\end{align*}
Further, as $X_W$ depends continuously on $W$, we also have 
\begin{align*}
  X^\ast \leftarrow_{i \rightarrow \infty} X^{(i_\ell+1)} = X_{W^{(i_\ell)}} \rightarrow_{i \rightarrow \infty} X^\ast_{W^\ast}.
\end{align*}
$(iv)$: We first assume that $\delta_{\inf} := \liminf_{i \rightarrow \infty} \delta^{(i)} > 0$. There are hence only finitely many steps with $\delta^{(i)} < \frac{1}{2} \delta_{\inf}$.
Since $\inf(A) > 0$, it follows that also $\gamma^{\ast} := \lim_{i \rightarrow \infty} \gamma^{(i)} > 0$. Further, as $\{\gamma^{(i)}\}_{i \in \N} \subset \Theta$, 
there necessarily exists an $n \in \N$ such that $\gamma^{(i)} = \gamma^{\ast}$ for all $i \geq n$. 
Thus, there exists a subsequence $\{X^{(i_\ell)}\}_{\ell \in \N}$, $i_\ell \geq n$, $\ell \in \N$, for which 
\begin{align}\label{distancepropmat}
 \|X^{(i_\ell)} - S\| \geq \frac{1}{2} \delta_{\inf} > 0,
\end{align}
for all $\ell \in \N$ and all $S \in S_{\gamma^\ast}^\ast$. As by $(ii)$ however $X^{(i)}$ remains bounded, $\{X^{(i_\ell)}\}_{\ell \in \N}$ must have an accumulation point, which by $(iii)$
is within $S_{\gamma^{\ast}}^\ast$. This is in direct contradiction to \cref{distancepropmat}, and we obtain that $\delta_{\inf} = 0$ must instead hold true.
Further, as $A$ is bounded and $\inf(\Theta) = 0$, this also implies $\gamma^\ast = 0$.
By construction, there hence exist subsequences $\{X^{(i_\ell)}\}_{\ell \in \N}$ (given through the steps in which $\gamma^{(i_\ell+1)} < \gamma^{(i_\ell)}$) as well 
as $\{S^{(\ell)}\}_{\ell \in \N}$, $S_\ell \in \mathcal{S}^\ast_{\gamma^{(i_\ell)}}$, such that
\begin{align*}
 \|X^{(i_\ell)} - S_\ell\| \leq 2 \delta_{i_\ell} \leq 2 \inf(A)^{-1} \theta_{i_\ell}, \quad i \in \N.
\end{align*}
As $\theta_{i_\ell} \rightarrow 0$ follows by $\gamma^\ast = 0$, we obtain $\|X^{(i_\ell)} - S_\ell\| \rightarrow 0$. This was to be shown.
\end{proof}
\subsection{Examples and counterexamples}
Throughout this section, it suffices to consider the conventional \IRLS{}-$0$ algorithm, that is $s_i = 1$, $i \in \N_0$.
In \cref{convergenceexample}, we discusses different cases of convergence with particular regard to \cref{declinelemma}.
It also highlights, in contrast to part \textit{(iv)}, that if $\gamma^{(i)}$ decays too fast to $\gamma^\ast = 0$, the sequence $X^{(i)}$
may converge, but not to a limit of stationary point of $f_\gamma$.\\ 
The subsequent \cref{divergencelemma} then provides a rather rare case of divergence, and gives a counter example
for \cref{declinelemma}, part \textit{(ii)}, given $\gamma^\ast = 0$.\\
Regarding part \textit{(iii)}, it remains unclear whether $X^{(i)}$ can in fact have multiple
accumulation points, or if the assertion can be improved. Also part \textit{(iv)} makes the impression that it may be possible to derive a stronger implication, but proving or disproving this likewise remains subject to future research.
\begin{example}\label{convergenceexample}
 Let $\mathcal{L}: \R^{2 \times 2} \rightarrow \R^2$ be a linear operator and $y \in \R^2$ such that
 \begin{align*}
  \mathcal{L}^{-1}(y) = \{ X(a,b) \mid a,b \in \R \}, \quad X(a,b) := \begin{pmatrix}
                                                                       a & 1 \\
                                                                       1 & b
                                                                      \end{pmatrix}.
 \end{align*}
 The matrix $X(a,b)$ has rank $1$ if and only if $ab = 1$. While there is not a unique solution to the rank minimization problem, we have
\begin{align*}
 \mathcal{X}^\ast := \underset{X \in \mathcal{L}^{-1}(y),\ \mathrm{rank}(X) = 1}{\mathrm{argmin}} \sigma_1(X) = \left\{ \begin{pmatrix}
  1 & 1 \\
  1 & 1
 \end{pmatrix}, \begin{pmatrix}
  -1 & 1 \\
  1 & -1
 \end{pmatrix} \right\}.
\end{align*}
The only stationary points of $f_\gamma$, $\gamma \in [0,1)$, in turn are given by
\begin{align*}
\mathcal{X}^{\gamma,\ast} := \left\{
 \begin{pmatrix}
  -\sqrt{1-\gamma} & 1 \\
  1 & -\sqrt{1-\gamma}
 \end{pmatrix}, \
  \begin{pmatrix}
  0 & 1 \\
  1 & 0
 \end{pmatrix}, \
  \begin{pmatrix}
  \sqrt{1-\gamma} & 1 \\
  1 & \sqrt{1-\gamma}
 \end{pmatrix} \right\},
\end{align*}
where the second one is repellent.
 For $X^{(i)} = X(a_i,b_i) = X_{W_{\gamma,X^{(i-1)}}}$, $i \in \N$, we have the rational functions (cf. \cref{rationalupdate}) $a_i = q_1(a_{i-1},b_{i-1})$ and $b_i = q_2(a_{i-1},b_{i-1})$ with
 \begin{align*}
  q_1(a,b) = \frac{a+b}{1 + \gamma + b^2}, \quad q_2(a,b) = \frac{a+b}{1 + \gamma + a^2}.
 \end{align*}
 Short calculations then show that 
 \begin{align*}
  0 < q_1(a,b) \cdot q_2(a,b) \leq 1 \quad & \Leftrightarrow \quad a + b \neq 0 \\
  q_1(a,b) = 0 \quad \Leftrightarrow \quad q_2(a,b) = 0 \quad & \Leftrightarrow \quad a + b = 0.
 \end{align*}
 Thus, the stationary point $a = b = 0$ is either reached directly or never.
 Further, for any $a,b \in \R$, it holds true that 
 \begin{align}
  |q_1(a,b) - q_2(a,b)| \leq |a-b|, \label{contractive}
 \end{align}
 where equality requires $a = b$, or $\gamma = 0$ and $ab = 1$.
The only attracting fixedpoints $(a,b) = (q_1(a,b),q_2(a,b))$ are given through the three cases
\begin{align*}
  a = b = 0, \quad & \mbox{ if } 1 \leq \gamma,\\
  a = b = \sqrt{1 - \gamma}, \quad & \mbox{ if } 0 < \gamma < 1,\\
  ab = 1, \quad & \mbox{ if } \gamma = 0.
\end{align*}
The fact that here the lowest norm solution (cf. \cref{gammainf}) is a local maximum of $f_\gamma$, $0 < \gamma < 1$, is however not representative for the general situation, 
but rather coincidentally holds true.
The behavior of $X^{(i)}$ now greatly depends on the sequence $\{\gamma^{(i)}\}_{i \in \N}$.
Without loss of generality, we assume $0 < a_1 b_1 \leq 1$, $a_1 \geq b_1$, in the following three cases.
\paragraph{(i)} Firstly, for $\gamma^\ast := \lim \gamma^{(i)} > 0$, we know by \cref{declinelemma} that $X^{(i)}$ will converge to either one of the two attracting fixedpoints in $\mathcal{X}^{\gamma,\ast}$.
\paragraph{(ii)} For $\gamma = 0$, we have
\begin{align*}
 0 < ab < 1 \quad \Rightarrow \quad |q_1(a,b)| > |a|\ \wedge\ |q_2(a,b)| > |b|,
\end{align*}
whereby the sequence $X^{(i)}$ given $\gamma^{(i)} \equiv 0$ will converge from below to a rank $1$ matrix with 
\begin{align*}
 a_1^2 + b_1^2 + 2 \leq \|\lim_{i \rightarrow \infty} X^{(i)}\|_F^2 \leq (a_1 - b_1)^2 + 4,
\end{align*}
where the second inequality follows due to \cref{contractive} and $a_i b_i \leq 1$, $i \in \N$. 
The norm of the limit can thus be arbitrarily large, yet a single sequence never diverges.
\paragraph{(iii)} Given $\gamma^\ast := \lim \gamma^{(i)} = 0$, the rate of decline is deciding.\label{para:toofastdecl}
Let $1 < s < a$ (or for that matter $0 < b < \frac{1}{s}$) as well as $\mu > 0$. Then given
\begin{align*}
 0 < \gamma = \Gamma(a,b) := \mu \cdot (a/s-1 + b/s - b^2),
\end{align*}
it is $q_1(a,b) > s \Leftrightarrow \mu < 1$. Thus, for fixed $0 < \mu < 1$, $a_1 > s$ and 
\begin{align*}
 \gamma^{(i)} := \min(\gamma^{(i-1)},\Gamma(a_i,b_i)),
\end{align*}
it follows with \cref{contractive} that $a_i \rightarrow s$ and $b_i \rightarrow \frac{1}{s}$
as well as $\gamma_i \searrow 0$. On the other side, 
\begin{align*}
 \gamma^{(i)} = \min(\gamma^{(i-1)},\sqrt{1-\gamma^{(i-1)}} - |b_{i-1}|),
\end{align*}
will always yield a sequence for which $X^{(i)}$ converges to a point in $\mathcal{X}^\ast$ (cf. \cref{declinelemma}, part $(iv)$).
\end{example}
\begin{proposition}\label{divergencelemma}
 Let, as in \cref{divexample}, $\mathcal{L}: \R^{2 \times 3} \rightarrow \R^4$ be a linear operator and $y \in \R^4$ such that
 \begin{align*}
  \mathcal{L}^{-1}(y) = \{ X(a,b) \mid a,b \in \R \}, \quad X(a,b) := \begin{pmatrix}
                                                                       a & 1 & a + 1 \\
                                                                       b + 1 & b & b + 1
                                                                      \end{pmatrix}.
 \end{align*}
 Assume now that $\gamma^{(i)} \equiv 0$, $i \in \N$, and $X^{(1)} = X(a_1,b_1)$ for $1 < a_1$ and $0 < b_1 < \frac{1}{a_1 + \frac{1}{2a_1}}$. 
 Then
 \begin{align*}
  \sigma_1(X^{(i)}) \rightarrow \infty, \quad \sigma_2(X^{(i)}) \rightarrow 0.
 \end{align*}
 The iterate thus diverges, but comes arbitrarily close to the set of rank $1$ matrices. In particular, it is
 $X^{(i)} = X(a_i,b_i)$ for $a_{i-1} < a_i$ and $0 < b_i < \frac{1}{a_i + \frac{1}{2a_i}}$, $i \in \N$, with
 $(a_i,b_i) \rightarrow (\infty,0)$.
\end{proposition}
Experiments show that $X^{(i)}$ also diverges similarly for other starting values and $\gamma^{(i)} \searrow 0$.
However, given the canonical starting value $X^{(0)}$ as in \cref{gammainf}, $X^{(i)}$ will converge to the global minimizer given through $(a,b) = (-1,-2/3)$.
Interestingly (cf. \cref{divexample}), the function $f_\gamma$ for the problem setting in \cref{divergencelemma} also happens to exhibit a diverging sequence of stationary points for $\gamma \searrow 0$.
Whether this property or the fact that $\sigma_r(X^{(i)})$ converges to zero is in general related to the possibility of a diverging sequence $X^{(i)}$ however remains unclear.
\begin{proof}
It is $X^{(i)} = X(a_i,b_i)$ where $a_i$ and $b_i$ are rational polynomials dependent on both $a_{i-1}$ and $b_{i-1}$ as carried out in \cref{polyincrlemma}.
With the properties shown therein, it follows that $\sigma_1(X^{(i)}) \rightarrow \infty$ and $\sigma_2(X^{(i)}) \rightarrow 0$.
\end{proof}
Alternatively, the previous can be shown in a less elementary based on \cref{declinelemma} and \cref{divexample} which only requires to prove $a_i > 1$ and $b_i > 0$
(in \cref{polyincrlemma}, these are parts $(i)-(iii)$). 
\section{Alternating iteratively reweighted least squares (\AIRLS{})}\label{reloptmatrix}
For large matrices, it becomes a computational burden to maintain the equality $\mathcal{L}(X) = y$.
In the following, we consider a relaxation of such as it allows for data sparse algorithms to be applied that require to violate that exact constraint. 
\subsection{Relaxation of affine constraint}\label{fagamma}
Let $a_\gamma: \R \rightarrow \R$, $a_\gamma(s) := s - n \log(\gamma)$, $\gamma > 0$. As these function are monotonically increasing, a composition with such does not change minimizers. We correspondingly define
\begin{align*}
 f_\gamma^a(X) := a_\gamma \circ f_\gamma(X) = \log( \prod_{i = 1}^{\infty} 1 + \frac{\sigma_i(X)^2}{\gamma} ), \quad J^{a}_{\gamma}(X,W) := a_\gamma \circ J_\gamma(X,W).
\end{align*}
Likewise, we also have $f^a_\gamma(X) = f^{(2)}_\gamma(X) - m\log(\gamma)$.
For an appropriate, constant scaling factor $c_{\mathcal{L}} > 0$ and $\omega > 0$, we can then
weaken the affine constraint $\mathcal{L}(X) = y$ into an additional penalty term,
\begin{align}
 F_{\gamma,\omega}^{a}(X) & := \| \mathcal{L}(X) - y \|_F^2 + c_{\mathcal{L}}\cdot \omega^2 \cdot f_\gamma^a(X), \nonumber \\
 \mathcal{J}^{a}_{\gamma,\omega}(X,W) & := \| \mathcal{L}(X) - y \|_F^2 + c_{\mathcal{L}}\cdot \omega^2 \cdot J^{a}_{\gamma}(X,W). \label{relaxedJ}
\end{align}
The updates of weight matrices will be the same as for the constraint, original version.
We are particularly interested in the limit $\omega \rightarrow 0$ and desire to obtain asymptotically identical updates for the iterate $X$. And indeed, by \cref{omegaredlemma}, we have that
\begin{align*}
 \lim_{\omega \rightarrow 0} \underset{X \in \R^{n_1 \times n_2}}{\mathrm{argmin}}\ \mathcal{J}^{a}_{\gamma,\omega}(X,W) & = \underset{X \in \mathcal{L}^{-1}(y)}{\mathrm{argmin}}\ J^{a}_{\gamma}(X,W).
\end{align*}
However, this argument is of course weaker than rigorous statements about local convergence.
While the objective function $F^a_{\gamma,\omega}$ is also decreased when lowering $\omega$, this does no longer hold true for $\gamma$.
However, this problem can be circumvented by making $\omega$ dependent of $\gamma$ in terms of $\omega_\gamma := \sqrt{\gamma}$. Since $\frac{\partial}{\partial \gamma} \gamma \cdot (\log(s^2 + \gamma) - \log(\gamma)) \geq 0$, we obtain
\begin{align*}
 \frac{\partial}{\partial \gamma} F_{\gamma,\omega_\gamma}^{a}(X) = c_{\mathcal{L}} \cdot \frac{\partial}{\partial \gamma} (\gamma \cdot f_{\gamma}^{a}(X)) \geq 0.
\end{align*}
In the following, we accordingly skip the index $\omega$. 
\begin{corollary}\label{mondeclt}
We consider a modified version of \cref{alg:irlsmr}, in which the updates are instead derived from a minimization of the relaxed $\mathcal{J}^a_\gamma(X,W)$, \cref{relaxedJ}. 
Then still\,\footnote{The other assertions of \cref{declinelemma} may (likely) hold true as well, but this remains subject to future research.}
 $F^a_{\gamma^{(i)}}(X^{(i)}) \leq F_{\gamma^{(i-1)}}^{a}(X^{(i-1)})$,
holds true for all $i \in \N$.
\end{corollary}
\begin{proof}
 The argumentation is analogous as $\frac{\partial}{\partial \gamma} (\gamma \cdot f^a_{\gamma}(X)) \geq 0$.
\end{proof} 
\subsection{Data sparse, alternating optimization}\label{altopt}
Given that the aim of optimization is the minimization of the rank, it is reasonable to restrict the optimization to the variety
\begin{align*}
 V_{\leq R} = \{ X \in \R^{n \times m} \mid \mathrm{rank}(X) \leq R \} = \{ X = YZ \mid Y \in \R^{n \times R}, \ Z \in \R^{R \times m} \}
\end{align*}
of at most rank $R$ matrices. The value $R$ may be chosen adaptively or as the bound
\begin{align}\label{boundR}
 R = 1 + \max \{ r \in \N \mid \mathrm{dim}(V_{\leq r}) \leq \ell \} = 1 + \lfloor \frac{1}{2}(n+m - \sqrt{(n+m)^2 - 4\ell}) \rfloor,
\end{align}
where $\mathrm{dim}(V_{\leq r}) = (n+m)r - r^2$. With this choice, the degrees of freedom in the utilized low rank representations in fact exceeds the number of measurements, so $R$ is always large enough (and based only on available information).
Instead of directly optimizing $X$ (and the weight $W$), its sought components $Y$ and $Z$ can be treated alternatingly, as we similarly did in our prior works \cite{GrKr19_Sta,Kr20_Tre}\footnote{The derivation of weights within these works follows a different, more heuristic approach based on perturbed representations, but the resulting algorithms are (quite interestingly) similar.} under the name \SALSA{} (stable ALS approximation). Therein, it is also reasoned why this and related methods essentially are immune to overestimation of $R$. However, the approach requires the relaxation of the constraint $\mathcal{L}(X) = y$ as in \cref{fagamma}.
Further, if the order of complexity $\mathcal{O}(nm)$ is to be avoided, then the operator $\mathcal{L}$ must exhibit some form of simpler representation as well. In the following, we therefore assume that 
there exists a low rank $r_L$ decomposition of the operator $\mathcal{L}(X) = L \mathrm{vec}(X)$ itself in form of 
\begin{align*}
 L = \sum_{i = 1}^{r_L} L_{2,i}^T \circledast L_{1,i} \in \R^{\ell \times nm}, \quad L_{1,i} \in \R^{\ell \times n},\ L_{2,i} \in \R^{m \times \ell},\quad i = 1,\ldots,r_L,
\end{align*}
where $\circledast$ is the row-wise Khatri-Rao product. Given this decomposition, it follows that
% \begin{align*}
 $\mathcal{L}(Y Z) = \sum_{i = 1}^{r_L} (L_{1,i} Y) \odot (Z L_{2,i})^T$, for all $Y \in \R^{n \times R}$, $Z \in \R^{R \times m}$,
% \end{align*}
 where $\odot$ is the Hadamard product.
 \begin{remark}
  For sampling operators (cf. \cref{sec:numexp}), such a decomposition is trivially achieved for $r_L = 1$. 
  Further, due to separability, the rows of $Y$ and columns of $Z$ can be treated independently, that is, in the calculation of \cref{Yupdate,Zupdate}.
 \end{remark}
 In the following, let $X = U_R \Sigma_R V_R^T$ be the reduced version of the SVD $X = U \Sigma V^T$, for $U_R \in \R^{n \times R}$, $\Sigma_R \in \R^{R \times R}$, $V_R \in \R^{m \times R}$.
Utilizing the ambiguity within the representation of $X$, one can always achieve that either $Y = U_R$ or $Z = V_R^T$ without changing $X$.
In the first case, the update in $Z$ is determined by
\begin{align}\label{simpLYZ}
 & \ \underset{Z \in \R^{R \times m}}{\mathrm{argmin}}\ \mathcal{J}^a_{\gamma}(YZ,W) = \underset{Z \in \R^{R \times m}}{\mathrm{argmin}} \| \mathcal{L}(Y Z) - y \|_F^2 + c_{\mathcal{L}} \cdot \gamma \cdot \|W^{1/2} Y Z \|_F^2.
 \end{align}
Under given circumstances, this term can be simplified significantly. Due to its analogous derivation, we state the following for general $p \in [0,1]$.
\begin{lemma}
 Let $Y = U_R$ and $X = U_R \Sigma_R V_R^T$. Then
 \begin{multline}
 Z_{\gamma,Y,\Sigma_R} :=  \ \underset{Z \in \R^{R \times m}}{\mathrm{argmin}}\ \mathcal{J}^a_{\gamma}(Y Z,W_{\gamma,X}) \\ \label{Zupdate}
 = \underset{Z \in \R^{R \times m}}{\mathrm{argmin}} \
 \big\| \big( \sum_{i = 1}^{r_L} L^T_{2,i} \circledast (L_{1,i} Y ) \big) \mathrm{vec}(Z) - y \big\|^2_F + c_{\mathcal{L}} \gamma  \big\| (\Sigma_R^2 + \gamma I_R)^{-1/2+p/4} Z \big\|_F^2.
 \end{multline}
\end{lemma}
\begin{proof}
 On the one hand, we can simplify
 \begin{align*}
  \| \mathcal{L}(Y Z) - y \|_F = \| \big(\sum_{i = 1}^{r_L} L^T_{2,i} \circledast (L_{1,i} Y ) \big) \mathrm{vec}(Z) - y \|_F.
 \end{align*}
On the other hand, with 
\begin{align*}
 W_{\gamma,X} = U (\Sigma \Sigma^T + \gamma I)^{-1+p/2} U^T = U_R (\Sigma_R^2 + \gamma I_R)^{-1+p/2} U^T_R + \gamma^{-1+p/2} U_R^{\perp} (U_R^{\perp})^T
\end{align*}
it follows that
\begin{align}\label{simpWYZ}
 \|W^{1/2} Y Z \|_F = \| U_R (\Sigma_R^2 + \gamma I)^{-1/2+p/4} U_R^T U_R Z \| = \| (\Sigma_R^2 + \gamma I_R)^{-1/2+p/4} Z \|,
\end{align}
where we used that $U = (U_R, U_R^{\perp}) \in \R^{n \times (R + (n-R))}$ with $U_R^T U_R^{\perp} = 0$.
\end{proof}
In the second case, $Z = V^T$, aboves simplification analogously works for \cref{simpLYZ},
but \cref{simpWYZ} requires the same modification of the augmented objective function as in \cref{sec:mirror}.
Therefor, the component $Y$ and a corresponding weight matrix $W^{(2)} \in \R^{m \times m}$, constrained by $W^{(2)} = (W^{(2)})^T \succ 0$, are instead updated as minimizers of
\begin{align*}
 \mathcal{J}^{2,a}_{\gamma}(X,W^{(2)}) := a_\gamma \circ \left( \mathrm{trace}(W^{(2)} (X^T X + \gamma I_m)) - \log \det(W^{(2)}) - m \right).
\end{align*}
This is the mirrored version of the biased, original choice $\mathcal{J}^a_{\gamma}(X,W)$ (which corresponds to $\mathcal{J}^{1,a}_{\gamma}(X,W^{(1)})$). 
As the update in $W$ then yields $W^{(2)}_{\gamma,X} = V (\Sigma^T \Sigma + \gamma I)^{-1} V^T$ (cf. \cref{sec:mirror}), an analogous simplification as in \cref{simpWYZ} now follows due to symmetry of the problems with respect to transposition of $X$.
\begin{corollary}
 Let $Z = V_R^T$ and $X = U_R \Sigma_R V_R^T$. Then
   \begin{multline}
  Y_{\gamma,\Sigma_R,Z} :=  \ \underset{Y \in \R^{n \times R}}{\mathrm{argmin}}\ \mathcal{J}^{2,a}_{\gamma}(Y Z,W^{(2)}_{\gamma,X}) \label{Yupdate} \\
 =  \ \underset{Y \in \R^{n \times R}}{\mathrm{argmin}} \| \big(\sum_{i = 1}^{r_L} (Z L_{2,i})^T \circledast L_{1,i} \big) \mathrm{vec}(Y) - y \|^2_F + c_{\mathcal{L}} \gamma  \| Y (\Sigma_R^2 + \gamma I_R)^{-1/2+p/4} \|_F^2.
  \end{multline}
\end{corollary}
Following the aboves scheme as summarized in \cref{alg:datasparsematrix}, it is thus possible to alternatingly optimize $Y$ and $Z$ (as well as their respective weight matrices) with the same neglectable computational complexity as plain alternating least squares, compared to the otherwise required order imposed by the size $n m$ of $X$. 
Further, \cref{mondeclt} still holds true for $X^{(i)} := Y^{(i)} Z^{(i)}$, whereby a monotonic decrease of the objective function is guaranteed. Though \cite{FoRaWa11_Low,MoFa12_Ite} also consider more effective updates or variations of conventional \IRLS{}, in both cases one still iterates over the full matrix $X$.
More specific convergence results for \AIRLS{} on the other hand are subject to future work. 
 \begin{algorithm}
  \caption{\AIRLS{}-$p$}
  \begin{algorithmic}[1] \label{alg:datasparsematrix}
  \STATE set $Y^{(0)}$ (arbitrary), $Z^{(0)}$ (row-orthogonal), $\gamma^{(0)} > 0$  
  \FOR{$i = 1,2,\ldots$}
  \STATE{calculate the SVD $U_R \Sigma_R \widetilde{V}^T := Y^{(i-1)}$}
  \STATE{replace $Y^{(i-1)} := U_R$ and update $Z^{(i)} := Z_{\gamma^{(i-1)},U_R,\Sigma_R}$ (see \cref{Zupdate})}
  \STATE{calculate the SVD $\widetilde{U} \Sigma_R V_R^T := Z^{(i)}$}
  \STATE{replace $Z^{(i)} := V_R^T$ and update $Y^{(i)} := Y_{\gamma^{(i-1)},\Sigma_R,V_R^T}$  (see \cref{Yupdate})}
  \STATE{set $\gamma^{(i)} \leq \gamma^{(i-1)}$}
  \ENDFOR
  \end{algorithmic}
\end{algorithm}%
\section{Numerical experiments}\label{sec:numexp}
For simplicity, we choose $n = m$ in all numerical experiments and thereby only consider square matrices. As mentioned in the introduction, \cref{sec:intro}, we are mainly interested in the ARM problem itself, which is why the criteria for success are as laid out in \cref{sec:expsetup}. The analogous analysis of the briefly considered ACM is not further elaborated, as only $\sigma(X) \in \R^n$ needs to be replaced with the absolute values of the vector $x \in \R^n$.
For the \textsc{Matlab} code behind all results, please contact the author.
\subsection{Reference solutions, measurements vectors and operators}
Each measurement vector is constructed via a (not necessarily sought for) rank $r \in \N$ reference solution, 
which in turn relies on a randomly generated low rank decomposition,
\begin{align*}
 y = \mathcal{L}(X^\rs) \in \R^\ell, \quad X^\rs  = Y^\rs Z^\rs \in \R^{n \times m}, \quad Y^\rs \in \R^{n \times r}, \ Z^\rs \in \R^{r \times m}.
\end{align*}
All entries of the two components $Y^\rs$ and $Z^\rs$ are assigned independent, normally distributed entries. 
For the operator $\mathcal{L}$, we distinguish between two types.
\paragraph{Gaussian measurements}\label{gaussop}
With a Gaussian measurements, we refer to
 $\mathcal{L}(X) := L\ \mathrm{vec}(X)$,
generated via a random matrix $L \in \R^{\ell \times nm}$
with independent, normally distributed entries.
\paragraph{Random sampling operator}\label{samplop}
As sampling operator, we denote
  $\mathcal{L}(X) := \{ X_{p_i} \}_{i = 1}^\ell$,
 for uniformly drawn indices $\{p_1,\ldots,p_\ell\} \subset \{1,\ldots,n\} \times \{1,\ldots,m\}$. We however repeat the random number generation until at least $r$ entries of each column and each row of the corresponding matrix structure are observed.
\subsection{Solution methods}\label{sec:solmeth}
Based on a sufficiently large starting value $\gamma^{(0)} > 0$, we choose $\gamma^{(i)} = \nu \gamma^{(i-1)}$,
where $\nu < 1$ remains constant throughout each single run of an algorithm. If
not otherwise specified, the default weight strength, as it is our main interest, is given through $p = 0$. Note that due to convex nature for $p = 1$, \IRLS-$1$ always finds the one nuclear norm minimizer $\mathrm{argmin}_{X \in \mathcal{L}^{-1}(y)} \|X\|_{\ast}$ if only the sequence $\{\gamma^{(i)}\}_{i \in \N_0}$ declines reasonably slow.
% Though due to the shared breaking criteria between all \IRLS-$p$ implementations, sometimes a slightly lower rate $\nu$ is necessary (cf. \cref{res1}), as the algorithm might otherwise stop to save futile effort judging from the general case $p < 1$.
We consider the following three types of optimization. 
\paragraph{Full, image based (\textsc{IRLS}-$p$)}
As in \cref{X_W}, the full matrix is optimized based on the (literally interpreted)
image update formula. When instability threatens to occur, 
the equivalent kernel based update formula \cref{kernelbased} for $X_0 = X^{(0)}$ is used instead.
\paragraph{Full, relaxed}
The relaxed constraints described in \cref{fagamma} are utilized, resulting in non equivalent updates compared to the image or kernel method. In this case, the residual $\|\mathcal{L}(X) - y\|$ is expected to converge to $0$ parallel to the decline of $\gamma$, but this is not guaranteed.
\paragraph{Alternating (\textsc{AIRLS}-$p$)}
The relaxed objective function subject to alternating between complementary weights is used, minimized by the alternating \cref{alg:datasparsematrix} as introduced in \cref{altopt}. The rank of the applied variety is fixed as the non-restrictive bound reasoned in \cref{boundR},
dependent only on the available problem size.
\subsection{Experimental setup and evaluation}\label{sec:expsetup}
The reference solution $X^\rs$ is not necessarily the sought for solution. So in order to evaluate to output $X^{(\mathrm{alg})}$,
in order to evaluate the output $X^\alg$ we compare their \textit{non-neglectable} singular values. We therefor define
\begin{align*}
 \mathrm{det}^2_{n,\gamma,\varepsilon}(X) & := \gamma^{n-\mathrm{rank}_{\varepsilon}(X)} \prod_{i = 1}^{\mathrm{rank}_{\varepsilon}(X)} (\sigma_i(X)^2 + \gamma), 
\end{align*}
for $\mathrm{rank}_{\varepsilon}(X) := \max \{i \in \{1,\ldots,n\} \mid \sigma_i(X) > \epsilon \cdot \|X\|_F \}$, $\epsilon := 10^{-6}$.
We firstly examine the residual norm, secondly compare the approximate ranks, and lastly compare the products of singular values. 
The latter two aspects are reflected by the limit
\begin{align*}
 \mathcal{Q}_\varepsilon(X^{(\mathrm{alg})},X^\rs) := 
 \lim_{\gamma \searrow 0} \frac{\mathrm{det}_{n,\gamma,\varepsilon}(X^{(\mathrm{alg})})}{\mathrm{det}_{n,\gamma,\varepsilon}(X^\rs)} 
  \in [0,0.98] \cup (0.98,1.005) \cup [1.005,\infty].
\end{align*}
The three intervals are related to the categorization into improvements, successes or fails as outlined below, whereas
the limits $0$ or $\infty$ are reached if and only if $\mathrm{rank}_{\varepsilon}(X^{(\mathrm{alg})})$ and $\mathrm{rank}_{\varepsilon}(X^\rs)$ differ.
\paragraph{Post iteration}\label{para:postiter}
As the truncation of minor singular values is a numerical necessity,
additional care has been taken when this process may falsify the results.
Therefor, we alternatingly and sufficiently often project\footnote{While this process is suitable for already approximately low rank matrices, 
it does not work well as a standalone algorithm, apart from the fact that the rank is generally unknown.} any $X^\alg$ in question to the sets $\mathcal{L}^{-1}(y)$ 
and $V_{\leq r}$, where $r$ is then chosen as rank of $X^\rs$.
This usually allows to reduce the parameter $\varepsilon$ to machine precision.
\paragraph{Details of comparison}
If $\|\mathcal{L}(X^{\alg}) - y\| > 10^{-6} \|y\|$ or if for the quotient, it holds $\mathcal{Q}_\varepsilon(X^{\alg},X^\rs) = \infty$, then the result is considered a \textit{strong failure}.
If $\|\mathcal{L}(X^{\alg}) - y\| \leq 10^{-6} \|y\|$, then on the one hand we refer to $1.005 \leq \mathcal{Q}_\varepsilon(X^{\alg},X^\rs) < \infty$ as \textit{weak failure}.
On the other, for $0.98 < \mathcal{Q}_\varepsilon(X^{\alg},X^\rs) < 1.005$, we consider the result \textit{successful}, while for $\mathcal{Q}_\varepsilon(X^{\alg},X^\rs) \leq 0.98$,
 we say the result is an \textit{improvement}, subject to the consideration above.
\paragraph{Sensitivity analysis}
We lower the meta parameter $\nu = \nu_k = \sqrt{\nu_{k-1}}$ (cf. \cref{sec:solmeth}), starting with $\nu_0 = 1.2$, and rerun the respective algorithm from the start until the result is not a \textit{failure}.
However, after too many reruns $k > k_{\max}$, we give up and thus either achieve a \textit{weak} or \textit{strong} \textit{failure} depending on the result for $k = k_{\max}$.
All other meta parameters for each algorithm are common to all respective experiments.
\subsection{Presentation of results}\label{sec:presres}
Each experiment is reflected upon in three different ways as summarized in \cref{metatable}.
\paragraph{ACM/ARM/recovery tables}
For each instance, we list the percentual numbers of ACM/ARM improvements, successes or fails as defined in \cref{sec:expsetup}. Successes are further distinguished regarding recoveries, that is whether $\|X^\alg - X^\rs\|_F \leq 10^{-4} \|X^\rs\|_F$. In near all cases where this is fulfilled, the relative residual even falls below $10^{-6}$ (see \cref{sec:visres}), in which case the algorithm stops automatically\footnote{Needless to say, this is the only point at which the reference solution itself is used within the algorithm, and only done in order to save unnecessary computation time.}.
Note that both improvements as well as fails with respect to ASRM naturally nearly exclude recoveries with accuracy $10^{-4}$, and always so for $10^{-6}$.
\paragraph{ACM/ARM/recovery figures}
More distinguished visualizations of the results underlying the above mentioned tables can be found in \cref{sec:visres} as described therein.
\paragraph{$\gamma$-decline sensitivity}
A depiction of results regarding the sensitivity analysis outlined in \cref{sec:expsetup} is covered in \cref{sec:visres} as well.
\def\externaltablecaption#1#2#3{\label{#1}table as specified in \cref{sec:presres} for \cref{#2} (see \cref{#3} for more details)}
\subsection{Affine rank minimization}\label{sec:exp1}
\begin{experiment}\label{exp1}
 For $n = 12$, $r = 3$ and $\ell \in \{63,72,108\}$, we consider
 the ARM problem based on \textit{Gaussian measurements} under varying choice of the 
 weight strength $p \in \{0,0.2,0.4,0.6,0.8,1\}$ utilizing \textit{full, image based} updates. Each constellation is repeated $1000$ times for $k_{\max} = 12$.
 The results are covered in \cref{tab1,res1,res1b}.
\end{experiment}%
%%%
\externaltable{htb!}{table_2_matrix_by_p}{\externaltablecaption{tab1}{exp1}{res1b}}%
%%%
The degrees of freedom of at most rank $r$ square matrices of size $n$ is $\dim(V_{\leq r}) = 2 n r - r^2$, which in \cref{exp1} results in $\dim(V_{\leq r}) = 63$. For $\ell = 63$, \IRLS{}-$p$ (at least for small $p$) quite often yields improvements (which is surprisingly different to the vector case as discussed in \cref{sec:exp0}), but despite the minimal number of measurements also often manages to reconstruct the reference solution.
In particular the sensitivity analysis for $\ell = 72$ demonstrates (in accordance with \cite{MoFa12_Ite}) that $p = 0$ seems to overall be the
best choice in this setting, whereas the most significant differences can be observed between $p = 0.6$ and $p = 1$. 
\subsection{Observing the theoretical phase transition for matrix recovery}\label{sec:phasetrans}
\begin{experiment}\label{exp10}
 For $n = 12$, $r = 3$ and $\ell \in \{62,63,64\}$, we consider
 the ARM problem based on \textit{samples} or \textit{Gaussian measurements} for the 
 weight strength $p = 0$ utilizing \textit{full, image based} updates. Each constellation is repeated $1000$ times for the increased value $k_{\max} = 14$.
 The results are covered in \cref{tab10}.
\end{experiment}%
\externaltable{htb!}{table_n1_matrix_conj}{\externaltablecaption{tab10}{exp10}{res10b}}%
As in \cref{exp1}, we have $\dim(V_{\leq r}) = 63$. From \cref{tab10}, we can clearly observe that $\ell = 62$ are as expected too few measurements to allow for any recovery\footnote{as no reference solution happens to minimize the product of singular values}.
In turn, the value $\ell = \dim(V_{\leq r}) + 1$ (which here is $\ell = 64$)
is the minimal sufficient amount of \textit{generic}\footnote{To be more precise, \textit{generic} in that context is an algebraic property that is stronger than the ones that stem from analysis or probability theory, but roughly similar.} measurements (thus not including sampling) to provide $\mathcal{L}^{-1}(\mathcal{L}(X^\rs)) \cap V_{\leq r_\rs} = \{ X^\rs \}$ for \textit{generic} $X^\rs \in V_{\leq r_\rs}$, as more generally proven in \cite{BrGeMiVa21_Alg}. 
Indeed, there are only multiple solutions for $\ell \leq 63$, while for $\ell = 64$,
the exceptions displayed in \cref{tab10} fail to withstand the post-iteration (cf. \cref{para:postiter}) and thus do in fact not approximate truly rank $r$ matrices.
\subsection{Alternating, affine rank minimization}\label{sec:exp2}
\begin{experiment}\label{exp2}
 For $n = 12$, $r = 3$ and $\ell \in \{63,72,108\}$, we consider
 the ARM problem based on \textit{samples} or \textit{Gaussian measurements} utilizing either \textit{full, image based} or \textit{relaxed} updates (with weight strength $p = 0$). In the sample case, we additionally consider \textit{alternating} optimization.
 Each constellation is repeated $1000$ times for $k_{\max} = 12$. The results are covered in \cref{tab2,res2,res2b}.
\end{experiment}%
%%%
\externaltable{htb!}{table_3_matrix_by_setting}{\externaltablecaption{tab2}{exp2}{res2b}}%
%%%
Clearly, sample based ARM and the related matrix completion pose more difficult problems. 
Interestingly, it seems as in some cases, an utterly slow decay of $\gamma$ may yet provide successful solutions (in stark contrast to the vector case, see \cref{res0}), though this behavior is less extreme for Gaussian measurements (see \cref{res2}). Note that in this setting, $k = 12$ corresponds to $\nu \approx 1.00004^{-1}$,
which will result in multiples of $100.000$ iterations\footnote{Considering that the starting value is the same canonical one for each rerun of the algorithm and that $\nu$ approaches values very close to $1$, it seems unlikely that multiple reruns produce better results also by chance. We have further confirmed this presumption in more specific tests.}. Further, the difference between image/kernel based or relaxed updates appears marginal, whereas \AIRLS{}-$0$ provides similar results. Both \IRLS{}-$0$ and its alternating version frequently improve upon the reference solution, but in that case, a recovery is naturally impossible. This gray zone in the case of completion problems is even more spread out (in $\ell$) than for Gaussian measurements, which is in accordance with general theory.
\begin{experiment}\label{exp3}
 For $n \in \{50,200,500,1000\}$, $r \in \{5,7\}$ and $\ell = c_{\mathrm{mf}} \dim(V_{\leq r}) = c_{\mathrm{mf}} (2nr - r^2)$, $c_{\mathrm{mf}} \in \{1.2,1.6,2\}$, we consider the ARM problem based on \textit{samples} approached via \textit{alternating} optimization (with weight strength $p = 0$). Each constellation is repeated $100$ times. The results are covered in \cref{tab3,tab3c,res3,res3b,res3c,res3d}.
\end{experiment}%
%%%
\externaltable{htb!}{table_4_alt_matrix_completion_r5_by_value}{\externaltablecaption{tab3}{exp3}{res3b}}
\externaltable{htb!}{table_5_alt_matrix_completion_r7_by_value}{\externaltablecaption{tab3c}{exp3}{res3d}}%
%%%
Overall, \AIRLS{}-$0$ seldomly provides a worse result than the reference solution, and only so for larger mode sizes and marginal measurements factors. For rank $r = 7$, hardly any ARM failure can be observed, while in both cases, a measurements factor of $c_{\mathrm{mf}} = 2$ provides a (near) perfect completion rate. Note here that failure to complete the reference solution is not the fault of any ARM algorithm, but $V_{\leq r} \cap \mathcal{L}^{-1}(y)$ may simply contain more than just $X^\rs$.
Enlarging the rank of the reference solution subject to a constant measurements factor seems to in turn provide easier problems. 
\section{Conclusions and outlook}
We have shown that while seemingly neglectable in practice, the possible degeneracy of 
some ARM problems may impose theoretical boundaries, as well as be related to 
diverging sequences of IRLS iterates.
We have embedded the distinguished structure of the log-det approach to ARM that underlies IRLS-$0$ into
a broader, nested minimization scheme, in order to 
prove global convergence properties. 
Further, we have proven how the convergence of IRLS-$0$ iterates to undesired solutions can be caused merely through a too fast decline of the regularization parameter $\gamma$. In numerical experiments, we have confirmed that even subject to a moderately fast
decrease of $\gamma$, IRLS-$0$ is the overall optimal choice. 
As further laid out, ACM and ARM exhibit particularly different behaviour that become most noticeable when the number of measurements tends to a minimal level.
Following the extension of local convergence results to allow for a switching between complementary weights,
we have presented an A(lternating)IRLS-$p$ method that directly minimizes the original log-det objective function.
This method is as data-sparse and low in computational complexity as other representation based approaches, yet it is immune to an overestimation of the rank. This work will be continued to the tensor setting, 
in which we generalize IRLS-$0$ to sum-of-ranks minimization, as well as AIRLS-$0$ to  hierarchical
tensor decompositions.

\appendix
\FloatBarrier
\section{(Remaining proofs)} \label{sec:minorproofs}
\begin{proof} (of \cref{nestedminimization})
Clearly, $A^\ast \subset A_{n+1} = D$. Let therefore $A^\ast \subset A_{k}$ be true for all $k = s+1,\ldots,n$
as well as $a^\ast \in A^\ast$, $\widetilde{a} \in A_{s+1}$. %If
Due to the nestedness condition \cref{assnested} towards the minimizers of $g_k$, $k = s+1,\ldots,n$, it inductively follows that
 \begin{align}\label{strongernested}
A_k = \{ a \in D \mid g_k(a) =  \min_{b \in D} g_k(b) \} \subset A_{k+1}, \quad k = s+1,\ldots,n. 
 \end{align}
We consider a sequence of minimizers $\{a_\gamma\}_{\gamma > 0}$ of $G_\gamma$ with $a_\gamma \rightarrow a^\ast$.
Then by definition
$G_\gamma(\widetilde{a}) - G_\gamma(a_\gamma) \geq 0$ for all $\gamma > 0$.
It thereby follows that
\begin{align*}
g_s(\widetilde{a}) - g_s(a_\gamma) \geq  R_\gamma(a_\gamma) - R_\gamma(\widetilde{a}) + \sum_{k = s+1}^n \gamma^{s-k} (g_k(a_\gamma) - g_k(\widetilde{a}))  ,
\end{align*}
for functions $R_\gamma(a_\gamma), R_\gamma(\widetilde{a}) \in \mathcal{O}(\gamma)$.
Since $\widetilde{a} \in A_{s+1} \subset \ldots \subset A_n$, and by \cref{strongernested} $g_k(a_\gamma) \geq g_k(\widetilde{a})$ for all $k = s+1,\ldots,n$, we have
$g_s(\widetilde{a}) \geq g_s(a_\gamma) + R_\gamma(a_\gamma) - R_\gamma(\widetilde{a})$.
Taking the limit $\gamma \searrow 0$ thus yields
\begin{align*}
 g_s(\widetilde{a}) \geq g_s(a^\ast).
\end{align*}
Since both $\widetilde{a} \in A_{s+1}$ and $a^\ast \in A^\ast$ were arbitrary, this shows $A^\ast \subset A_s$, which was to be shown.
Inserting $\widetilde{a} = a^\ast$ further yields that $\gamma^{s-k} (g_k(a_\gamma) - g_k(a^\ast)) \searrow 0$ for $\gamma \searrow 0$
as the lefthand side vanishes.
In particular, this also implies that 
\begin{align*}
 |g_k(a_\gamma) - g_k(a^\ast)| \leq c \gamma^{k-s}, \quad k = s+1,\ldots,n.
\end{align*}
This finishes the proof.
\end{proof}

\begin{proof} (direct version for \cref{convergencelemma})
Let $\overline{X} = \lim_{\gamma \searrow 0} X_\gamma$. Assume that there exists a rank $r$ matrix $X^\ast \in \mathcal{L}^{-1}(y)$ with $\prod_{i = 1}^r \sigma_i(X^\ast) = q_0 \prod_{i = 1}^r \sigma_i(\overline{X})$ for $q_0 < 1$. Thereby, there is (a unique) $\alpha > 0$ with $\prod_{i = 1}^r \sigma_i(X^\ast)^2 = \prod_{i = 1}^r (\sigma_i(\overline{X})-2 \alpha)^2$. Then for all $X \in \mathcal{L}^{-1}(y)$ with $\|X - \overline{X}\|_F \leq \alpha$
it holds true that
\begin{align*}
 q_\gamma := \frac{\exp f_\gamma(X^\ast)}{\exp f_\gamma(X)} & = \prod_{i = 1}^r \frac{\sigma_i(X^\ast)^2 + \gamma}{\sigma_i(X)^2 + \gamma} \prod_{i = r + 1}^n \frac{\gamma}{\sigma_i(X)^2 + \gamma} \leq 
 \prod_{i = 1}^r \frac{\sigma_i(X^\ast)^2 + \gamma}{(\sigma_i(\overline{X}) - \alpha)^2} \underset{\gamma \searrow 0}{\longrightarrow} q_0,
\end{align*}
where we used that $(\sigma_i(X) - \sigma_i(\overline{X}))^2 \leq \sum_{i = 1}^n (\sigma_i(X) - \sigma_i(\overline{X}))^2 \leq \|X - \overline{X}\|^2_F \leq \alpha^2$.
Thus, there exists $\Gamma_1 > 0$ such that $q_\gamma < 1$ for all $0 < \gamma \leq \Gamma_1$. Since however $X_\gamma$ converges to $\overline{X}$, there is $\Gamma_2 > 0$ for which $\|X_\gamma - \overline{X}\|_F \leq \alpha$ for all $0 < \gamma \leq \Gamma_2$. This is in direct contraction to $f_{\gamma}(X_\gamma) = \min_{X \in \mathcal{L}^{-1}(y)} f_{\gamma}(X) \leq f_{\gamma}(X^\ast)$.
\end{proof}

\begin{lemma}\label{minimization}
Given $\ell \leq n \leq k$, let $L \in \R^{\ell \times n}$ and $H \in \R^{k \times n}$ both have full rank. The solution to $x^\ast := \underset{Lx = y}{\mathrm{argmin}}\ \frac{1}{2} \|Hx\|_F^2$ is 
  $x^\ast = (H^T H)^{-1} L^T (L (H^T H)^{-1} L^T)^{-1} y$.
Given a representation $K \in \R^{n \times n - \ell}$ of the kernel of $L$, the solution can also be written as
$x^\ast = x_0 - K(K^T H^T H K)^{-1} K H^T H x_0$,
where $x_0$ is one arbitrary solution to $L x_0 = y$.
\end{lemma}
\begin{proof}
 The corresponding Lagrangian is $\mathcal{L}(x,\lambda) = \frac{1}{2} \|Hx\|_F^2 - \lambda^T (Lx - y)$ with optimality
conditions
\begin{align*}
 \frac{\partial}{\partial \lambda} \mathcal{L}(x,\lambda) & = Lx - y \overset{!}{=} 0, \quad \frac{\partial}{\partial x} \mathcal{L}(x,\lambda) = H^T Hx - L^T \lambda \overset{!}{=} 0.
\end{align*}
The second equation yields $x = (H^T H)^{-1} L^T \lambda$. Inserted in the first one, we obtain
\begin{align*}
 Lx & = L (H^T H)^{-1} L^T \lambda = y \quad \Leftrightarrow \quad \lambda = (L (H^T H)^{-1} L^T)^{-1} y.
\end{align*}
Inserting this in turn yields
$x = (H^T H)^{-1} L^T (L (H^T H)^{-1} L^T)^{-1} y$.
For the second part, consider that $x^\ast = x_0 + \underset{v \in \R^{n - \ell}}{\mathrm{argmin}}\ \frac{1}{2} \|H(x_0 + Kv)\|_F^2$
involves an ordinary least squares problem which leads to the stated solution.
\end{proof}

\begin{lemma}\label{omegaredlemma}
 For $L \in \R^{\ell \times m}$, $\ell \leq m$, as well as $y \in \mathrm{image}(L)$ and an invertible matrix $B \in \R^{m \times m}$, let
  $x_\omega := \mathrm{argmin}_{x \in \R^m}\ \|Lx - y\|_F^2 + \omega^2 \|Bx\|_F^2$.
 Then
  $\lim_{\omega \rightarrow 0} x_\omega = \mathrm{argmin}_{x \in \R^m, \ Lx = y}\ \|Bx\|_F$.
\end{lemma}
\begin{proof}
Firstly, substituting $z = B x$ and $A = L B^{-1}$ yields
 $B x_\omega 
 = (A^T A + \omega^2 I)^{-1} A^T y$.
Now, given a SVD $A = U \Sigma V^T$, we obtain
\begin{align*}
 B x_\omega & = (V \Sigma^T \Sigma V^T + \omega^2 I)^{-1} V \Sigma^T U^T y = V (\Sigma^T \Sigma + \omega^2 I)^{-1} \Sigma^T U^T y \\
 & = V \diag((\sigma_1^2 + \omega^2)^{-1} \sigma_1,\ldots,(\sigma_m^2 + \omega^2)^{-1}  \sigma_m) U^T y \\
 & \underset{\omega \rightarrow 0}{\rightarrow} V \diag(\sigma_1^{-1},\ldots,\sigma_k^{-1},0,\ldots,0) U^T y, 
\end{align*} 
for $\mathrm{rank}(L) = k$. Thus, as the last line describes the pseudo-inverse of $A$, we have
\begin{align*}
 \lim_{\omega \rightarrow 0} x_\omega = B^{-1} \underset{z \in \R^m, \ Az = y}{\mathrm{argmin}}\ \|z\|_F = \underset{x \in \R^m, \ A (Bx) = y}{\mathrm{argmin}}\ \|Bx\|_F, 
\end{align*}
which, due to $AB = L$, provides the to be shown result. 
\end{proof}
\section*{Acknowledgments}
The author would like to thank Maren Klever and Lars Grasedyck for fruitful discussions,
as well as Paul Breiding and Nick Vannieuwenhoven for conversations on generic recoverability.

\bibliographystyle{siam}
\bibliography{IRLS_matrix_bibliography}

\begin{thebibliography}{10}

\bibitem{BaNoRe10_Onl}
{\sc L.~{Balzano}, R.~{Nowak}, and B.~{Recht}}, {\em Online identification and
  tracking of subspaces from highly incomplete information}, in 2010 48th
  Annual Allerton Conference on Communication, Control, and Computing
  (Allerton), 2010, pp.~704--711.

\bibitem{BrGeMiVa21_Alg}
{\sc P.~Breiding, F.~Gesmundo, M.~Micha\l{}ek, and N.~Vannieuwenhoven}, {\em
  Algebraic compressed sensing (in preparation)}.
\newblock 2021.

\bibitem{CaCaSh10_ASi}
{\sc J.-F. Cai, E.~J. Candès, and Z.~Shen}, {\em A singular value thresholding
  algorithm for matrix completion}, SIAM Journal on Optimization, 20 (2010),
  pp.~1956--1982.

\bibitem{CaRe09_Exa}
{\sc E.~J. Cand{\`e}s and B.~Recht}, {\em Exact matrix completion via convex
  optimization}, Foundations of Computational Mathematics, 9 (2009), p.~717.

\bibitem{CaTa10_The}
{\sc E.~J. Cand\`{e}s and T.~Tao}, {\em The power of convex relaxation:
  Near-optimal matrix completion}, IEEE Trans. Inf. Theor., 56 (2010),
  pp.~2053--2080.

\bibitem{CaWaBo08_Enh}
{\sc E.~J. Cand{\`e}s, M.~B. Wakin, and S.~P. Boyd}, {\em Enhancing sparsity by
  reweighted l1 minimization}, Journal of Fourier Analysis and Applications, 14
  (2008), pp.~877--905.

\bibitem{Ch08_Res}
{\sc R.~Chartrand and V.~Staneva}, {\em Restricted isometry properties and
  nonconvex compressive sensing}, Inverse Problems, 24 (2008), p.~035020.

\bibitem{ChWo08_Ite}
{\sc R.~{Chartrand} and {Wotao Yin}}, {\em Iteratively reweighted algorithms
  for compressive sensing}, in 2008 IEEE International Conference on Acoustics,
  Speech and Signal Processing, 2008, pp.~3869--3872.

\bibitem{DaMiKe11_Sub}
{\sc W.~{Dai}, O.~{Milenkovic}, and E.~{Kerman}}, {\em Subspace evolution and
  transfer (set) for low-rank matrix completion}, IEEE Transactions on Signal
  Processing, 59 (2011), pp.~3120--3132.

\bibitem{DaDeFoGu10_Ite}
{\sc I.~Daubechies, R.~DeVore, M.~Fornasier, and C.~S. Güntürk}, {\em
  Iteratively reweighted least squares minimization for sparse recovery},
  Communications on Pure and Applied Mathematics, 63 (2010), pp.~1--38.

\bibitem{DaRo16_AnO}
{\sc M.~A. {Davenport} and J.~{Romberg}}, {\em An overview of low-rank matrix
  recovery from incomplete observations}, IEEE Journal of Selected Topics in
  Signal Processing, 10 (2016), pp.~608--622.

\bibitem{Fa02_Mat}
{\sc M.~Fazel}, {\em Matrix Rank Minimization with Applications}, PhD thesis,
  Stanford University, Stanford, CA, 2002.

\bibitem{FaHiBo03_Log}
{\sc M.~{Fazel}, H.~{Hindi}, and S.~P. {Boyd}}, {\em Log-det heuristic for
  matrix rank minimization with applications to hankel and euclidean distance
  matrices}, in Proceedings of the 2003 American Control Conference, 2003.,
  vol.~3, June 2003, pp.~2156--2162 vol.3.

\bibitem{FoRaWa11_Low}
{\sc M.~Fornasier, H.~Rauhut, and R.~Ward}, {\em Low-rank matrix recovery via
  iteratively reweighted least squares minimization}, SIAM Journal on
  Optimization, 21 (2011), pp.~1614--1640.

\bibitem{GiRoKo19_Alt}
{\sc P.~V. {Giampouras}, A.~A. {Rontogiannis}, and K.~D. {Koutroumbas}}, {\em
  Alternating iteratively reweighted least squares minimization for low-rank
  matrix factorization}, IEEE Transactions on Signal Processing, 67 (2019),
  pp.~490--503.

\bibitem{GoMa11_Con}
{\sc D.~Goldfarb and S.~Ma}, {\em Convergence of fixed-point continuation
  algorithms for matrix rank minimization}, Foundations of Computational
  Mathematics, 11 (2011), pp.~183--210.

\bibitem{GrKr19_Sta}
{\sc L.~Grasedyck and S.~Kr\"amer}, {\em Stable als approximation in the
  tt-format for rank-adaptive tensor completion}, Numerische Mathematik, 143
  (2019), pp.~855--904.

\bibitem{Gr11_Rec}
{\sc D.~Gross}, {\em Recovering low-rank matrices from few coefficients in any
  basis}, IEEE Transactions on Information Theory, 57 (2011), pp.~1548--1566.

\bibitem{HaMaLeZa15_Mat}
{\sc T.~Hastie, R.~Mazumder, J.~D. Lee, and R.~Zadeh}, {\em Matrix completion
  and low-rank svd via fast alternating least squares}, J. Mach. Learn. Res.,
  16 (2015), pp.~3367--3402.

\bibitem{KeMoOh10_Mat2}
{\sc R.~H. {Keshavan}, A.~{Montanari}, and S.~{Oh}}, {\em Matrix completion
  from a few entries}, IEEE Transactions on Information Theory, 56 (2010),
  pp.~2980--2998.

\bibitem{KeMoOh10_Mat}
{\sc R.~H. Keshavan, A.~Montanari, and S.~Oh}, {\em Matrix completion from
  noisy entries}, J. Mach. Learn. Res., 11 (2010), pp.~2057--2078.

\bibitem{Kr20_Tre}
{\sc S.~Kr\"amer}, {\em {T}ree tensor networks, associated singular values and
  high-dimensional approximation}, dissertation, RWTH Aachen University,
  Aachen, 2020.
\newblock Ver\"offentlicht auf dem Publikationsserver der RWTH Aachen
  University; Dissertation, RWTH Aachen University, 2020.

\bibitem{LeBr10_Ato}
{\sc K.~{Lee} and Y.~{Bresler}}, {\em Admira: Atomic decomposition for minimum
  rank approximation}, IEEE Transactions on Information Theory, 56 (2010),
  pp.~4402--4416.

\bibitem{LeSe05_Non}
{\sc A.~S. Lewis and H.~S. Sendov}, {\em Nonsmooth analysis of singular values.
  part i: Theory}, Set-Valued Analysis, 13 (2005), pp.~213--241.

\bibitem{MaAr14_Ite}
{\sc D.~{Malioutov} and A.~{Aravkin}}, {\em Iterative log thresholding}, in
  2014 IEEE International Conference on Acoustics, Speech and Signal Processing
  (ICASSP), 2014, pp.~7198--7202.

\bibitem{MoFa10_Rew}
{\sc K.~{Mohan} and M.~{Fazel}}, {\em Reweighted nuclear norm minimization with
  application to system identification}, in Proceedings of the 2010 American
  Control Conference, 2010, pp.~2953--2959.

\bibitem{MoFa12_Ite}
{\sc K.~Mohan and M.~Fazel}, {\em Iterative reweighted algorithms for matrix
  rank minimization}, Journal of Machine Learning Research, 13 (2012),
  pp.~3441--3473.

\bibitem{Na95_Spa}
{\sc B.~K. Natarajan}, {\em Sparse approximate solutions to linear systems},
  SIAM Journal on Computing, 24 (1995), pp.~227--234.

\bibitem{NgKiSh19_Low}
{\sc L.~T. {Nguyen}, J.~{Kim}, and B.~{Shim}}, {\em Low-rank matrix completion:
  A contemporary survey}, IEEE Access, 7 (2019), pp.~94215--94237.

\bibitem{OyMoFaHa11_Asi}
{\sc S.~{Oymak}, K.~{Mohan}, M.~{Fazel}, and B.~{Hassibi}}, {\em A simplified
  approach to recovery conditions for low rank matrices}, in 2011 IEEE
  International Symposium on Information Theory Proceedings, 2011,
  pp.~2318--2322.

\bibitem{RaKr99_Ana}
{\sc B.~D. {Rao} and K.~{Kreutz-Delgado}}, {\em An affine scaling methodology
  for best basis selection}, IEEE Transactions on Signal Processing, 47 (1999),
  pp.~187--200.

\bibitem{Re11_ASi}
{\sc B.~Recht}, {\em A simpler approach to matrix completion}, Journal of
  Machine Learning Research, 12 (2011), pp.~3413--3430.

\bibitem{ReFaPa10_Gua}
{\sc B.~Recht, M.~Fazel, and P.~Parrilo}, {\em Guaranteed minimum-rank
  solutions of linear matrix equations via nuclear norm minimization}, SIAM
  Review, 52 (2010), pp.~471--501.

\bibitem{ReXuHa11_Nul}
{\sc B.~Recht, W.~Xu, and B.~Hassibi}, {\em Null space conditions and
  thresholds for rank minimization}, Mathematical Programming, 127 (2011),
  pp.~175--202.

\bibitem{TaWe13_Nor}
{\sc J.~Tanner and K.~Wei}, {\em Normalized iterative hard thresholding for
  matrix completion}, SIAM Journal on Scientific Computing, 35 (2013),
  pp.~S104--S125.

\bibitem{Va13_Low}
{\sc B.~Vandereycken}, {\em Low-rank matrix completion by riemannian
  optimization}, SIAM Journal on Optimization, 23 (2013), pp.~1214--1236.

\bibitem{XuZiWeZh12_Ana}
{\sc Y.~Xu, W.~Yin, Z.~Wen, and Y.~Zhang}, {\em An alternating direction
  algorithm for matrix completion with nonnegative factors}, Frontiers of
  Mathematics in China, 7 (2012), pp.~365--384.

\end{thebibliography}
\endgroup
% 
%%%% supplementary material
% 
\newpage
\setcounter{section}{0}
\setcounter{equation}{0}
\setcounter{figure}{0}
\setcounter{table}{0}
\setcounter{page}{1}
\makeatletter
\renewcommand{\theequation}{SM\arabic{equation}}
\renewcommand{\thefigure}{SM\arabic{figure}}
\renewcommand{\thetable}{SM\arabic{table}}
\renewcommand{\thesection}{SM\arabic{section}}
\renewcommand{\thepage}{SM\arabic{page}}
\renewcommand{\thetheorem}{SM\arabic{theorem}}
\renewcommand{\thesubsection}{\thesection.\arabic{subsection}}
\renewcommand{\thesubsubsection}{\thesubsection.\arabic{subsubsection}}
\renewcommand{\theparagraph}{\thesubsubsection.\arabic{paragraph}}
\renewcommand{\thesubparagraph}{\theparagraph.\arabic{subparagraph}}
\makeatother

% \supplementary

\begin{center}
    \textbf{\MakeUppercase{Supplementary Materials:}}
% \MakeUppercase{Asymptotic Log-Det Rank Minimization Via (Alternating) Iteratively Reweighted Least Squares}}
\end{center}
\section{Affine cardinality minimization}\label{sec:exp0}
\begin{experiment}\label{exp0}
 For $n = 160$, $|\mathrm{supp(x^{(ss)})}| = 40$ and $\ell \in \{65,75,85,95\}$, we consider
 the cardinality minimization problem based on Gaussian measurements under varying choice of the 
 weight strength $p \in \{0,0.2,0.4,0.6,0.8,1\}$ utilizing image based updates. Each constellation is repeated $1000$ times. The results are covered in \cref{tab0,res0,res0b}.
\end{experiment}%
\def\externaltablecaption#1#2#3{\label{#1}table as specified in \cref{sec:presres} for \cref{#2} (see \cref{#3} for more details)}%
\externaltable{htb!}{table_1_vector_by_p}{\externaltablecaption{tab0}{exp0}{res0b}}%
The results suggest that $p = 0$ (or at least a very small value) is here an overall optimal choice as well. This is in apparent contradiction with the results of \cite{DaDeFoGu10_Ite}, who have also considered a very similar problem setting.
Their results seem to yield that $p = 0.5$ is recommendable, whereas for lower values, convergence is not always observed. This might however be due to the different and possibly less tolerant strategy chosen for the adaption\footnote{In \cite{DaDeFoGu10_Ite}, these parameters are denoted $\tau$ and $\varepsilon$, respectively.} of $\gamma$ that does not seem suitable for smaller values of $p$ (cf. \cref{convergenceexample}).
With $\ell$ close to $\mathrm{dim}(V_S) = |\mathrm{supp(x^\rs)}| = 40$, it is not suprising that fewer references solutions are recovered, though it should theoretically be possible as long as $\ell \geq 41$ (cf. \cite{BrGeMiVa21_Alg}).
In contrast to ARM in the matrix case, \IRLS{} here also tends to fail the ACM problem for minimal numbers of measurements, possibly due to the less fortunate, only partial nestedness of $\mathcal{V}_1$ (cf. \cref{vectorV1}). 
Note that for $\ell = 40$ in turn, finding a solution with cardinality less or equal $40$ immediately becomes trivial. 
\section{Visualization of numerical results}\label{sec:visres}
Each of the following even and odd numbered pair of pages contains two related visualizations of the results of one of \cref{exp0,exp10,exp1,exp2,exp3} as summarized in \cref{metatable}. These additional visualizations are constructed as described further below.
\begin{table}[H]
 \begin{center}
  \begin{tabular}{c||c|c||c}
   experiment & $\gamma$-sensitivity & ACM/ARM/recovery & -- table\\
   \hline\hline
   \cref{exp0} & \cref{res0} & \cref{res0b} & \cref{tab0} \\
   \hline
   \cref{exp1} & \cref{res1} & \cref{res1b} & \cref{tab1} \\
   \hline
   \cref{exp10} & \cref{res10} & \cref{res10b} & \cref{tab10} \\
   \hline   
   \cref{exp2} & \cref{res2} & \cref{res2b} & \cref{tab2} \\
   \hline
   \cref{exp3} $(\overline{r}_\rs = 5)$ & \cref{res3} & \cref{res3b} & \cref{tab3} \\
   \hline
   \cref{exp3} $(\overline{r}_\rs = 7)$ & \cref{res3c} & \cref{res3d} & \cref{tab3c} \\
  \end{tabular}
 \end{center}
\caption{\label{metatable}overview over experiments, related figures and tables}
\end{table}%
\paragraph{$\gamma$-decline sensitivity}\label{par:sensitivity}
To each single trial that did not yield a failure, 
we assign the one index $k$ for which the parameter $\nu = \nu_k$ first led to a successful 
or improving run as described in \cref{sec:expsetup}. The frequencies of these indices 
as well as fails are then plotted as bars, where improvements are plotted below the x-axis.
\paragraph{ACM/ARM/recovery figures}\label{par:recfigures}
We display the following points as \textit{button plot} (as defined below).
Given the $i$-th result $X^{(\mathrm{alg})}$ 
as well as reference solution $X^{(\mathrm{rs})}$, 
the x-value of the $i$-th point is given by the bounded quotient 
\begin{align*}
 x_i = \max(0.9,\min(\mathcal{Q}_\varepsilon(X^{(\mathrm{alg})},X^{(\mathrm{rs})}),1.05)),
\end{align*}
Each y-value is given by 
\begin{align*}
 y_i = \min(\|X^{(\mathrm{alg})} - X^{(\mathrm{rs})}\|_F/\|X^{(\mathrm{rs})}\|_F,1),
\end{align*}
Note that the algorithm stops automatically if that value falls below $10^{-6}$.
\paragraph{button plot}
With a button plot (with logarithmic scale in $y$), we refer to a two dimensional, 
clustered scatter plot. Therein, any circular markers with centers $(x_i,y_i)$ 
and areas $s_i$, $i = 1,\ldots,k$, that would (visually) overlap,
are recursively combined to each one larger circle $(\widehat{x},\widehat{y})$ 
with area $\widehat{s}$ according to the appropriately weighted means
\begin{align*}
\widehat{x} = \sum_{i = 1}^k \frac{s_i}{\widehat{s}} x_i,\quad \widehat{y} = \prod_{i = 1}^k y_i^{s_i/\widehat{s}},\quad \widehat{s} = \sum_{i = 1}^k s_i. 
\end{align*}
The centers of all resulting circles are indicated as crosses. Thus, if only one circle remains, then the position of that cross is given by the arithmetic mean of all initial x-coordinates and the geometric mean of all initial y-coordinates. If no disks are combined, then their centers are the initial coordinates and their areas are all equal.
\newpage
\section{Sensitivity and ARM/recovery figures}\mbox{}
\FloatBarrier
\def\barfigurecaption#1#2{\label{#1}Results for \cref{#2} as described in \cref{par:sensitivity}.}
\def\buttonfigurecaption#1#2{\label{#1}Results for \cref{#2} as described in \cref{par:recfigures}.}
\pdffigure{htb!}{bar_1_vector_by_p}{\barfigurecaption{res0}{exp0}} 
\pdffigure{htb!}{button_1_vector_by_p}{\buttonfigurecaption{res0b}{exp0}}
\pdffigure{htb!}{bar_2_matrix_by_p}{\barfigurecaption{res1}{exp1}}
\pdffigure{htb!}{button_2_matrix_by_p}{\buttonfigurecaption{res1b}{exp1}}
\pdffigure{htb!}{bar_n1_matrix_conj}{\barfigurecaption{res10}{exp10}}
\pdffigure{htb!}{button_n1_matrix_conj}{\buttonfigurecaption{res10b}{exp10}}
\pdffigure{htb!}{bar_3_matrix_by_setting}{\barfigurecaption{res2}{exp2}}
\pdffigure{htb!}{button_3_matrix_by_setting}{\buttonfigurecaption{res2b}{exp2}}
\pdffigure{htb!}{bar_4_alt_matrix_completion_r5_by_value}{\barfigurecaption{res3}{exp3}}
\pdffigure{htb!}{button_4_alt_matrix_completion_r5_by_value}{\buttonfigurecaption{res3b}{exp3}}
\pdffigure{htb!}{bar_5_alt_matrix_completion_r7_by_value}{\barfigurecaption{res3c}{exp3}}
\pdffigure{htb!}{button_5_alt_matrix_completion_r7_by_value}{\buttonfigurecaption{res3d}{exp3}}
\FloatBarrier
\section{Properties of rational functions in diverging IRLS sequence}
\begin{lemma}\label{polyincrlemma}
Let $B(a) := \frac{2a}{2a^2 + 1}$.
 For the rational polynomials $q_1$ and $q_2$ mentioned in the proof of \cref{divexample}, it holds true that
 \begin{align}
  1 \leq a \quad \mbox{ and } \quad 0 < b \leq B(a) \label{implcond}
 \end{align}
 implies
 \begin{align*}
  a < q_1(a,b) \quad \mbox{ and } \quad 0 < q_2(a,b) < B(q_1(a,b)).
 \end{align*}
 In particular, $(a^+,b^+) := (q_1(a,b),q_2(a,b))$ again fulfills \cref{implcond}
 and 
 \begin{align*}
  (a_1,b_1) := (2,1/ 3), \quad (a_{i+1},b_{i+1}) := (q_1(a_i,b_i),q_2(a_i,b_i)), \ i \in \N, 
 \end{align*}
 defines a sequence for which $(a_i,b_i) \rightarrow (\infty,0)$, $i \rightarrow \infty$.
\end{lemma}
\begin{proof}
 Straightforward calculation reveals that
 \begin{align*}
  q_1(a,b) = \frac{(4b+4)a^3+(-b^2+12b+8)a^2+(7b^2+20b+8)a-b^2-2}{(10b^2+8b+4)a^2+(2b^2+4)a+10b^2+16b+10} =: \frac{N_1(a,b)}{D_1(a,b)}
 \end{align*}
 and
 \begin{align*}
  q_2(a,b) = \frac{(6a+6)b^3+(-4a^2+18a+7)b^2+20ab+4a-4}{(10a^2+2a+10)b^2+(8a^2+16)b+4a^2+4a+10} =: \frac{N_2(a,b)}{D_2(a,b)}.
 \end{align*}
 First, we note that $D_1(a,b) > 0$ and $D_2(a,b) > 0$ for $a > 1$, $b > 0$.
 For any fixed $a > 0$, the (rational) polynomials
 \begin{alignat*}{2}
  & t_1: [0,\infty) \rightarrow [1,\infty), \quad && t_1(c) := 1 + c, \\
  & t_2: [0,\infty) \rightarrow (0,B(a)], \quad && t_2(s) := \frac{B(a)}{1 + s}
 \end{alignat*}
 are bijections that incorporate the boundary conditions \cref{implcond}. In the following, we prove that polynomials $P$ are of shape 
 \begin{align}
  P(c,s) = \sum_{i = 0}^{I} \sum_{j = 0}^{J} k_{i,j} c^i s^j, \quad \mbox{ for } \quad k_{i,j} \geq 0,\ k_{0,0} > 0, \label{ispositive}
 \end{align}
 for $i = 0,\ldots,I$ and $j = 0,\ldots,J$ (where $k$ depends on $P$), which implies that $P(c,s) > 0$ for $c,s > 0$.
 We now show the three inequalities in analogous ways:\\
 \textit{(i)}: The inequality $a < q_1(a,b)$ is equivalent to $F_1(a,b) := N_1(a,b) - a D_1(a,b) > 0$, whereas 
 \begin{align*}
  F_1(t_1(c),t_2(s)) =: \frac{P_1(c,s)}{Q_1(c,s)}
 \end{align*}
 is a rational polynomial for which $Q_1(c,s) := (2c^2+4c+3)^2(s+1)^2$ and $P_1(c,s)$ are of shape \cref{ispositive}. Thus, $a < q_1(a,b)$,
 and in particular $a \neq q_1(a,b)$ under the given boundary conditions.\\
 \textit{(ii)}: For $0 < q_2(a,b)$, it suffices to show that $N_2(a,b) > 0$.
 This follows as likewise 
 \begin{align*}
  N_2(t_1(c),t_2(s))  =: \frac{P_2(c,s)}{Q_2(c,s)}
 \end{align*}
 is a rational polynomial for which $Q_2(c,s) := (2c^2+4c+3)^3(s+1)^3$ and $P_2(c,s)$ are of shape \cref{ispositive}.\\
 \textit{(iii)}: For the third inequality, it remains to show that $F_3(a,b) := B(q_1(a,b)) - q_2(a,b) > 0$. 
 Again, 
 \begin{align*}
  F_3(t_1(c),t_2(s))  =: \frac{P_3(c,s)}{Q_3(c,s)}
 \end{align*}
 is a rational polynomial for which (appropriately chosen) $Q_3(c,s)$ and $P_3(c,s)$ are of shape \cref{ispositive}.\\
 \textit{(iv)}: For the last part, we use that the sequence $\{a_i\}_{i \in \N}$ as defined in \cref{polyincrlemma} is monotonically increasing. Assume now that $\{a_i\}_{i \in \N}$ is bounded. Then the sequence necessarily converges to some $a^\ast > 1$. Further,
 there exists a subsequence $\{b_{i_\ell}\}_{\ell \in \N}$ of $b_i$ that converges to some $b^\ast$. 
 We then obtain $a^\ast := \lim_{\ell \rightarrow \infty} a_{i_\ell+1} = \lim_{\ell \rightarrow \infty} a_{i_\ell}
 = \lim_{\ell \rightarrow \infty} q_1(a_{i_\ell},b_{i_\ell}) = q_1(a^\ast,b^\ast)$,
 where 
 \begin{align*}
  0 \leq b^\ast = \lim_{\ell \rightarrow \infty} b_{i_\ell} \leq \lim_{\ell \rightarrow \infty} B(a_{i_\ell}) = B(a^\ast).
 \end{align*}
 As $N_1(a^\ast,0) - a^\ast D_1(a^\ast,0) = 4(a^\ast-1)^2+6(a^\ast-1) > 0$, it follows together with $(i)$ that $q_1(a^\ast,b^\ast) > a^\ast$. This is a contraction which implies that $a_i$ must diverge.
\end{proof}

\end{document}